\documentclass[graybox]{svmult}

\usepackage[utf8]{inputenc}
\usepackage{amsmath}
\usepackage{amsfonts}
\usepackage{amssymb}
\usepackage{mathptmx}       
\usepackage{helvet}         
\usepackage{courier}        
\usepackage{type1cm}        
\usepackage{makeidx}         
\usepackage{graphicx}        
\usepackage{multicol}        
\usepackage[bottom]{footmisc}

\newcommand{\sumtwo}[2]{\sum_{\substack{#1 \\ #2}}} 
\newcommand{\abs}[1]{\left| #1\right|}

\newcommand{\calD}{\mathcal{D}}
\newcommand{\calE}{\mathcal{E}}
\newcommand{\calF}{\mathcal{F}}

\newcommand{\calH}{\mathcal{H}}

\newcommand{\calL}{\mathcal{L}}
\newcommand{\calM}{\mathcal{M}}

\newcommand{\calS}{\mathcal{S}}
\newcommand{\calT}{\mathcal{T}}

\newcommand{\fra}{\mathfrak{a}}

\newcommand{\frq}{\mathfrak{q}}

\newcommand{\frs}{\mathfrak{s}}

\newcommand{\frv}{\mathfrak{v}}
\newcommand{\frw}{\mathfrak{w}}

\newcommand{\bbA}{\mathbb{A}}

\newcommand{\bbC}{\mathbb{C}}
\newcommand{\bbD}{\mathbb{D}}
\newcommand{\bbE}{\mathbb{E}}

\newcommand{\bbN}{\mathbb{N}}

\newcommand{\bbP}{\mathbb{P}}
\newcommand{\bbQ}{\mathbb{Q}}
\newcommand{\bbR}{\mathbb{R}}
\newcommand{\bbS}{\mathbb{S}}

\newcommand{\bbZ}{\mathbb{Z}}

\newcommand{\sfe}{{\sf e}}

\newcommand{\ueta}{\underline{\eta}}
\newcommand{\ungamma}{\underline{\gamma}}

\newcommand{\ufrv}{\underline{\frv}}
\newcommand{\ufrw}{\underline{\frw}}
\newcommand{\Zd}{\bbZ^d}
\newcommand{\Rd}{\bbR^d}
\newcommand{\Kq}{{\mathbf K}^q}
\newcommand{\Ka}{{\mathbf K}^a}
\newcommand{\Kql}[1]{{\mathbf K}^q_{#1}}
\newcommand{\Kal}[1]{{\mathbf K}^a_{#1}}

\newcommand{\setof}[2]{\left\{#1 \,:\, #2 \right\}}

\newcommand{\fcone}{Y^>_\delta(h)}

\newcommand{\bcone}{Y^<_\delta(h)}

\newcommand{\lb}{\left(}
\newcommand{\rb}{\right)}
\newcommand{\lbr}{\left\{}
\newcommand{\rbr}{\right\}}

\newcommand{\dd}{{\rm d}}

\newcommand{\tq}[1]{t^\omega_{#1}}
\newcommand{\fq}[1]{f^\omega_{#1}}
\newcommand{\rrq}[1]{r^\omega_{#1}}
\newcommand{\fxq}[1]{f^{\theta_x \omega}_{#1}}

\newcommand{\ta}[1]{{\mathbf t}_{#1}}
\newcommand{\fa}[1]{{\mathbf f}_{#1}}

\newcommand{\qS}[2]{\calS_{#1}^{#2}} 
\newcommand{\aS}[1]{{\mathbf S}_{#1}} 
\newcommand{\cHm}[1]{\calH^-_{#1}}

\newcommand{\1}{\boldsymbol{1}}
\newcommand{\smo}[1]{{\mathrm o}\lb #1\rb }

\newcommand{\df}{\stackrel{\Delta}{=}}
\newcommand{\eqvs}{\stackrel{\sim}{=}}
\newcommand{\leqs}{\lesssim}            
\newcommand{\xpr}{x^\prime}

\newcommand{\Peff}{\mathrm{P_{\text{eff}}}}
\newcommand{\be}{\begin{equation}}
\newcommand{\ee}{\end{equation}}
\renewcommand{\emptyset}{\varnothing}

\makeindex             

\begin{document}

\title*{Stretched Polymers in Random Environment}
\author{Dmitry Ioffe and Yvan Velenik}
\institute{Dmitry Ioffe \at Department of Industrial Engineering and Management, Technion, Israel \email{ieioffe@ie.technion.ac.il}
\and Yvan Velenik \at Department of Mathematics, University of Geneva, Switzerland \email{Yvan.Velenik@unige.ch}}
\maketitle

\abstract{We survey recent results and open questions 
on the ballistic phase of stretched polymers in both annealed and
quenched random environments.}

\bigskip
\hfill\emph{This paper is dedicated to Erwin Bolthausen}

\hfill\emph{on the occasion of his 65th birthday}

\section{Introduction}
Stretched polymers or drifted random walks in random potentials could be considered either in their own right or as a more sophisticated and physically more realistic version of directed polymers. Indeed, directed polymers were introduced in~\cite{Kardar} as an effective SOS-type model for domain walls in Ising model with random ferromagnetic interactions~\cite{HenleyHuse}. Thus, directed polymers do not have overhangs or self-intersections, whereas  models of stretched polymers do not impose such constraints and in this respect, resemble ``real'' Ising interfaces.

Obviously stretched polymers inherit all the pending questions which are still open for their directed counterparts; in particular  a general mathematical description of the strong disorder regime is still missing.  It is perhaps unreasonable to expect that these issues would be easier to settle in the stretched context. However, attempts to analyze the model give rise to other more amenable issues of an intrinsic interest. To start with, even the annealed model is non-trivial in the stretched case.  Furthermore, both quenched and annealed models of stretched polymers exhibit a sub-ballistic to ballistic transition in terms of the pulling force which leads to a rich morphology of the corresponding phase diagram to be explored.  Finally, models of stretched polymers do  not have a natural underlying martingale structure, which rules out an immediate application of martingale techniques which played such a prominent role in the analysis of directed polymers (see e.g.~\cite{Bolthausen, Song, CarmonaHu, CSY1} as well as the review~\cite{CometsYoshida} and references therein).  It should be noted, however, that both an adjustment of the martingale approach (as based on, e.g., \cite{McLeish} - see also  Subsection~\ref{ssec:Convergence} below)  and non-martingale methods developed in the directed context~\cite{Sinai, Vargas1, Vargas2, Lacoin} continue to be relevant tools for the stretched models as well.

In this paper we try to summarize the current state of knowledge about stretched polymers. A large deviation level investigation of the model was initiated in the continuous context by Sznitman (\cite{Sznitman-book} and references therein) and then adjusted to the discrete setup in~\cite{Zerner, Flury-LD}. The case of high temperature discrete Wiener sausage with drift was addressed in~\cite{Trachsler}. The existence of weak disorder in higher dimensions has been established first for on-axis directions in~\cite{Flury-LE} and then extended to arbitrary directions in~\cite{Zygouras-LE}.  The main input of the latter work was a proof of a certain mass-gap condition for the (conjugate - see below) annealed model at high temperatures. In fact, the mass-gap condition in question  holds for a general class of off-critical self-interacting polymers in attractive potentials at all temperatures~\cite{IoffeVelenik-Annealed}, which leads to a complete Ornstein-Zernike level analysis of 
 the off-critical annealed case. In the quenched case, such an analysis paves the way to a refined description of what we call below the very weak disorder regime~\cite{IoffeVelenik-AOP, IoffeVelenik-inprogress}, which yields a stretched counter-part of the results of~\cite{Bolthausen}. Finally, the approach of~\cite{Vargas1} and the fractional moment method of~\cite{Lacoin} were  adjusted in~\cite{Zygouras-StrongDisorder} for a study of strong disorder in low ($d=2,3$) dimensions.

The paper is organized as follows:  The rest of Section~1 is devoted to a precise mathematical definition of the model and to an explanation of the key notions. The large deviation level theory is exposed in Section~2.  Path decomposition as described in Section~3 is in the heart of our approach. It justifies an effective directed structure of stretched polymers in the ballistic regime.  In the annealed case, it leads to a complete description of off-critical ballistic  models, which is the subject of Section~4.  The remaining Section~5 (Weak disorder) and Section~6 (Strong disorder) are devoted to a description of the rather incomplete  state of knowledge for the quenched models in the ballistic regime.

\subsection{Class of Models}
\textbf{Polymers.} A polymer $\gamma = (\gamma_0, \dots, \gamma_n)$ is a nearest-neighbor trajectory on the integer lattice $\Zd$. Unless stressed otherwise, $\gamma_0$ is always placed at the origin.  The length of the polymer is $ \abs{\gamma} = n$ and its spatial extension is $X (\gamma ) = \gamma_n - \gamma_0$.  In the most general case neither the length nor the spatial extension are fixed.

\smallskip
\noindent
{\bf Random Environment.} The random environment is a collection $\lbr V(x)\rbr_{x\in\Zd}$ of non-degenerate non-negative i.i.d. random variables which are normalized by $0\in {\rm supp}(V)$.  In the sequel we shall tacitly assume that ${\rm supp}(V)$ is bounded. Limitations and extensions (e.g. $\bbE V^d <\infty $ or existence of  traps $\lbr V(x) = \infty\rbr$) will be discussed separately in each particular case.  Probabilities and expectations with respect to the environment are denoted with bold letters $\bbP$ and $\bbE$. The underlying probability space is denoted as 
$\lb \Omega , \calF , \bbP\rb$. 

\smallskip
\noindent
\textbf{Weights.} The reference measure $P (\gamma ) \df (2d)^{-\abs{\gamma}}$ is given by simple random walk weights. The most general polymer weights we are going to consider are quantified by three parameters:
\begin{itemize}
\item the inverse temperature $\beta\geq 0$;
\item the external pulling force $h\in\Rd$;
\item  the mass per step $\lambda\geq 0$.
\end{itemize}
\noindent
The random  quenched weights are given by
\be 
\label{eq:qweight}
q_{\lambda ,h }^\beta (\gamma ) \df \exp\Bigl\{  h\cdot X (\gamma ) -\lambda \abs{\gamma}
- \beta\sum_{1}^{\abs{\gamma}} V (\gamma_i )\Bigr\}  P (\gamma ).
\end{equation}
In the sequel, we shall drop the index $\beta$ from the notation, and we shall drop the indices $\lambda$ or $h$ whenever they equal zero. The corresponding deterministic annealed weights are given by
\be 
\label{eq:aweight}
a_{\lambda ,h }  (\gamma ) \df
\bbE q_{\lambda ,h }  (\gamma )  = 
\exp\lbr  h\cdot X (\gamma ) -\lambda \abs{\gamma}
- 
\Phi_\beta (\gamma ) 
\rbr P (\gamma ) , 
\end{equation}
where $\Phi_\beta (\gamma ) \df \sum_x \phi_\beta \bigl(\ell_\gamma(x) \bigr)$, with $\ell_\gamma(x)$ denoting the local time (number of visits) of $\gamma$ at $x$, and 
\begin{equation}  
\label{eq:phibeta}
\phi_\beta (\ell ) = -\log\bbE e^{-\beta \ell V} .
\end{equation}
Note that the annealed potential is attractive, in the sense that $\phi_\beta (\ell +m ) \leq \phi_\beta (\ell ) +\phi_\beta (m)$. 

\smallskip
\noindent
\textbf{Path Measures and Conjugate Ensembles.} 
There are two natural types of ensembles to be considered: Those with fixed polymer length $\abs{\gamma}= n$, and those  with fixed spatial extension $X (\gamma ) = x$ or, alternatively, with fixed $ h\cdot X(\gamma )= N$.  Accordingly, we define the quenched  partition functions by
\begin{equation}  
\label{eq:pf}
Q_\lambda  (x )  \df \sum_{X (\gamma )=x} q_\lambda (\gamma ) , \ \ 
Q_\lambda  (N)  \df \sum_{h\cdot  x  = N} Q_\lambda  (x )\ \ {\rm and}\ \ 
Q_n  (h ) \df \sum_{|\gamma | = n} q_h (\gamma ) , 
\end{equation}
and use $A_\lambda  (x) \df \bbE Q_\lambda  (x )$,  $A_\lambda  (N)$ and $A_n  (h )$ to denote their annealed counterparts. 

Of particular interest is the case of a polymer with fixed length $\abs{\gamma}=n$. We define the corresponding quenched and annealed path measures by
\begin{equation}  
\label{eq:nMeasures}
\bbQ_n^h (\gamma ) \df \1_{\lbr \abs{\gamma } = n\rbr} \frac{q_h (\gamma )}{Q_n  (h )}\ \ 
{\rm and}\ \ 
\bbA_n^h (\gamma ) \df \1_{\lbr \abs{\gamma } = n\rbr} \frac{a_h (\gamma )}{A_n  (h )} .
\end{equation}
Following~\eqref{eq:pf}, the probability distributions $\bbQ_\lambda^x , \bbA_\lambda^x  , \bbQ_\lambda^N$ and $\bbA_\lambda^N$ are defined in the obvious  way.

\subsection{Ballistic and Sub-Ballistic Phases.}
In both the quenched and the annealed setups there is a competition between the attractive potential and the pulling force $h$: For small values of $h$ the attraction wins and the polymer is sub-ballistic, whereas it becomes ballistic if $h$ is large enough. These issues were investigated on the level of Large Deviations first in the continuous context of drifted Brownian motion among random obstacles in~\cite{Sznitman-book} and then for the models we consider here in~\cite{Zerner, Flury-LD}. Such large deviation analysis, however, overlooks the detailed sample-path structure of polymers and, in particular, does not imply law of large numbers or even existence of limiting spatial extension (speed). The law of large numbers in the annealed case was established in~\cite{IoffeVelenik-Annealed} together with other more refined analytic properties of annealed 
polymer measures in the ballistic regime
.(see Definition~\ref{def:super} below).   As is explained below, in the regime of weak disorder the annealed law of large numbers implies the 
quenched law of large numbers with the same 
limiting macroscopic spatial extension. We record all these facts as follows:
\begin{theorem}
\label{thm:ballistic}
There exist compact convex sets $\Ka \subseteq  \Kq$ with non-empty interiors, 
$0\in {\rm int}\Ka$, such that:
\begin{enumerate}
 \item If $h\in{\rm int} \Ka$, \ respectively  $h\in{\rm int}\Kq$, then, for any $\epsilon>0$,
\begin{equation}  
\label{eq:ballisticIn}
\lim_{n\to\infty}\bbA_n^h\bigl( \bigl|\frac{X(\gamma )}{n} \bigr| > \epsilon\bigr) = 0, \ \text{respectively}
\ \ 
\lim_{n\to\infty}\bbQ_n^h\bigl( \bigl|\frac{X(\gamma )}{n} \bigr| > \epsilon\bigr) = 0\ \ \bbP \text{-a.s.,}
\end{equation}
exponentially fast in $n$~\cite{Zerner, Flury-LD}.
\item If $h\not\in\Ka$, then there exists $v = v^a(h ,\beta )\neq 0$, such that, for any $\epsilon>0$,
\begin{equation}  
\label{eq:ballisticOutA}
\lim_{n\to\infty}\bbA_n^h\bigl(\bigl| \frac{X(\gamma )}{n} - v\bigr| > \epsilon\bigr)  = 0,
\end{equation}
exponentially fast in $n$~\cite{Zerner, Flury-LD, IoffeVelenik-Annealed}.
\item If $h\not\in\Kq$, then there exists a compact set $0\notin \calM_h$ such that
\begin{equation}  
\label{eq:ballisticOutq}
\liminf_{n\to\infty}\bbQ_n^h\Bigl( {\rm d}\bigl( \frac{ {X(\gamma )}}{n}, 
\calM_h \bigr) \Bigr)   = 0\ \ \bbP\text{-a.s.},
\end{equation}
exponentially fast in $n$~\cite{Zerner, Flury-LD}. Furthermore, if the dimension $d\geq 4$, then for any $h\neq 0$ fixed the set $\calM_h = \lbr v^a(h ,\beta )\rbr$ as soon as $\beta $ is sufficiently small.
\end{enumerate}
\end{theorem}
As described in the next subsection, $\Ka$ and $\Kq$ are support sets for certain Lyapunov exponents (norms).
\begin{definition}
\label{def:super}
The phases corresponding to $h\not\in\Ka$ and, respectively, $h\not\in\Kq$ are called ballistic.  For $* \in \lbr a,q\rbr$ the drifts $h$ are called sub-critical (respectively critical and super-critical) 
if  $h\in{\rm int}{\mathbf K}^*$ (respectively $h\in\partial {\mathbf K}^*$ and 
$h\not\in {\mathbf K}^*$).  
\end{definition}
A general (i.e., without an assumption of weak disorder) characterization of the set $\calM_h$ is given in Lemma~\ref{lem:MhSet}.  The question whether quenched models in the ballistic phase, or equivalently,   quenched models at super-critical drifts 
satisfy a law of large numbers is open with an exception of the small noise higher dimensional case (see Lemma~11 below).

\begin{remark}
The above theorem and its far reaching refinements hold for a large class of annealed models with attractive interactions.  In particular, no moment  assumptions on $V$ are needed. For instance, both the set $\Ka$ is defined and the corresponding  results  hold in the case of pure traps $V\in \lbr 0,\infty \rbr$.  
\end{remark}

The critical cases $h\in \partial \Ka$ and, of course,  $h\in\partial \Kq$ are open. It is easy to see that $\Ka\subset \Kq$ for sufficiently low temperatures. It is, however,  an open question (which depends on dimension $d=2,3$ or $d\geq 4$) whether the sets of critical drifts coincide for moderate or small values of $\beta$.

The sub-critical case $h\in {\rm int}\Ka$ has been worked out by Sznitman in the context of drifted  Brownian  motion among random obstacles~\cite{Sznitman-book}; see also~\cite{Antal} for some results in the case of random walks.

\subsection{Lyapunov Exponents.}
The quenched and annealed Lyapunov exponents are defined via
\begin{equation}  
\label{eq:LE}
\frq_\lambda (x ) \df -\lim_{N\to\infty} \frac1N \log Q_\lambda  (\lfloor Nx\rfloor  )\ \ 
{\rm and}\ \ 
\fra_\lambda (x ) \df -\lim_{N\to\infty} \frac1N \log A_\lambda  (\lfloor Nx\rfloor  ) .
\end{equation}
\begin{theorem}
Both $\frq_\lambda$ and $\fra_\lambda$ are defined for all $\lambda\geq 0$. Moreover, for every $\lambda\geq 0$, $\frq_\lambda\geq \fra_\lambda$ and both are equivalent norms on $\bbR^d$ : there exist $c_\lambda^1, c_\lambda^2 \in (0, \infty )$ such that
\begin{equation}  
\label{eq:LEproperties}
c_\lambda^1 |x| \leq \fra_\lambda (x) \leq \frq_\lambda (x) \leq c_\lambda^2 |x| .
\end{equation}
In particular, $\frq_\lambda$ and $\fra_\lambda$ are support functions,
\[
 \frq_\lambda (x) = \max_{h\in\partial \Kql{\lambda}}h\cdot x\quad{\rm and}\quad
 \fra_\lambda (x) = \max_{h\in\partial \Kal{\lambda}} h\cdot x .
\]
 of compact convex sets $\Kal{\lambda}\subseteq \Kql{\lambda}$ 
with non-empty interior containing $0$. 
\end{theorem}
\begin{remark}
The annealed Lyapunov exponent is always defined.  The proof of the existence of the (non-random) quenched Lyapunov exponent in~\cite{Zerner} is based on sub-additive ergodic theorem and requires an $\bbE V^d <\infty$ assumption. 
This was relaxed  to $\bbE V <\infty$ in \cite{KMZ}. 
However, no moment assumptions (apart from $\bbP\lb V =\infty \rb$ being small) are needed to justify existence of quenched Lyapunov exponents in the very  weak disorder case in higher dimensions~\cite{IoffeVelenik-AOP}.
\end{remark}
The sets $\Kal{\lambda}$ and  $ \Kql{\lambda}$ can be described equivalently as the unit balls for the polar norms
\[
\frq^*_\lambda (h) = \max_{x\neq 0}\frac{ h\cdot x}{\frq_\lambda (x)}\ \ 
{\rm and, accordingly,}\ \ 
\fra^*_\lambda (h) = \max_{x\neq 0}\frac{ h\cdot x}{\fra_\lambda (x)} .
\]
The set $\Ka$, respectively $\Kq$, in Theorem~\ref{thm:ballistic} is given by $\Kal{0}$, respectively $\Kql{0}$.

\subsection{Very Weak, Weak and Strong Disorder.}
\label{ssec:weakstrong}
Given $\lambda \geq 0$ and $\beta\geq 0$,  we say that the disorder is weak if $\fra_\lambda = \frq_\lambda$ and strong otherwise.  
Note that this definition is slightly different from the one employed in the directed 
case \cite{CometsYoshida}. 
The condition of being \emph{very weak} is of a technical nature. It means that the dimension is $d\geq 4$ and that, given either a fixed value of $h\neq 0$ or of $\lambda >0$, the inverse temperature $\beta$ is sufficiently small.  More precisely, we need a validity of~\eqref{eq:WDBound} below, which enables a fruitful  $L_2$-type control of partition functions and related quantities. In particular, the disorder is weak if it is very weak~\cite{Flury-LE, Zygouras-LE, IoffeVelenik-AOP} and, furthermore, in the regime of very weak disorder, both a $\bbP$-a.s. LLN and a $\bbP$-a.s. CLT hold for the limiting macroscopic spatial extension~\cite{IoffeVelenik-AOP, IoffeVelenik-inprogress}.
As we explain in Subsection~\ref{ssec:LLN}, the LLN is inherited by quenched models in the weak disorder regime. However, contrary to the directed case~\cite{CometsCLT}, it is not known whether CLT holds whenever the ratio between quenched and annealed 
partition functions stays bounded away from zero. 
Furthermore, it is not known whether, in $d\geq 4$, the disorder is weak for all $\lambda >0$ as soon as $\beta$ is small. In particular, proving that $\Ka = \Kq$ for small $\beta$ remains an open problem. 

Under mild assumptions on the potential $V$ ($\setof{x}{V(x)=0}$ does not percolate and $\lim_{\beta\to\infty} \log\bbE e^{-\beta V}/\beta =0$), it  is easy to see~\cite{Zygouras-StrongDisorder} that, for a given $\lambda$, the disorder is strong as soon as $\beta$ is large enough. Such a result is well-known even in the original context of the Ising model with random interactions~\cite{Wouts}. It was recently proved~\cite{Zygouras-StrongDisorder}  that in $d=2,3$ the disorder is strong for any $\lambda > 0$ and $\beta >0$; a short proof of the case $d=2$ is given in Section~\ref{sec:strong}.  

Furthermore the approach of Vargas~\cite{Vargas1} was adjusted in~\cite{Zygouras-StrongDisorder} in order to show that in the regime of strong disorder quenched conjugate measures necessarily contain macroscopic atoms.

\section{Large Deviations}
\label{sec:LD}
The following result holds under the presumably technical assumption that $\bbE V^d <\infty$ in the quenched case, but in full generality in the annealed case.
\begin{theorem}
\label{thm:LD}
For any $h\in \bbR^d$, the rescaled spatial  polymer extension $X(\gamma )/n $ satisfies large deviation principles (with speed $n$) under both $\bbA_n^h$ and, $\bbP$-a.s., under $\bbQ_n^h$ with the corresponding (non-random) rate functions $J_a^h$ and $J_q^h$ given by
\begin{equation}  
\label{eq:LDFunctions}
\begin{split}
J_a^h (v) &= \max_{\lambda}\lbr \fra_\lambda (v) - \lambda\rbr +\lb \Lambda_a (h) - h\cdot v\rb, \\
J_q^h (v) &= \max_{\lambda}\lbr \frq_\lambda (v) - \lambda\rbr 
+\lb \Lambda_q (h) - h\cdot v\rb,
\end{split}
\end{equation}
where $\Lambda_a(h) \df \lim_{n\to\infty}\frac1n\log A_n (f)$ and $\Lambda_q(h) \df \lim_{n\to\infty}\frac1n\log Q_n (f)$.
\end{theorem}
Let us explain Theorem~\ref{thm:LD}:  The following lemma shows that the sets $\Kal{\lambda}$ and $\Kql{\lambda}$ can be characterized as domains of convergence of certain power series.
\begin{lemma}
\label{lem:series}
\begin{enumerate}
\item For every $\lambda\geq 0$, if $h\in {\rm int}\Kal{\lambda}$ or, equivalently, 
if $\fra^*_\lambda (h) <1$, then
\begin{equation}  
\label{eq:hin}
\sum_x e^{h\cdot x} A_\lambda (x ) <\infty .
\end{equation}
\item For every $\lambda > 0$, 
if  $h\not \in \Kal{\lambda}$ or, equivalently, 
if $\fra^*_\lambda (h) >1$, then there exists $\alpha =   \alpha (h ) >0$ such that, all $n$ sufficiently large one can find $y = y_n$ satisfying:
\begin{equation}
\label{eq:hout}
e^{h\cdot y} A_{\lambda ,n} (y )  
\df
\sum_{X (\gamma ) = y, \abs{\gamma}= n} 
e^{h\cdot y}
a_\lambda (\gamma )
\geq e^{\alpha n} .
\end{equation}
\end{enumerate}
A completely analogous statement holds $\bbP$-a.s.\ in the quenched case.
\end{lemma}

We sketch the proof of Lemma~\ref{lem:series} at the end of the section. For the moment, let us assume its validity. For any $f\in\bbR^d$,
\[
e^{-\lambda n}A_n (f) = \sum_x e^{f\cdot x}A_{\lambda ,n} (x).
\]
Hence, by~\eqref{eq:hin}, 
$
\limsup_{n\to\infty}\frac1n \log A_n (f) \leq \lambda
$ 
whenever $f\in {\rm int}\Kal{\lambda}$, whereas~\eqref{eq:hout} implies that
$
\liminf_{n\to\infty}\frac1n\log A_n (f) \geq \lambda
$ 
for any $f\not\in\Kal{\lambda}$.  It is easy to see that  the strict inclusion
$\Kal{\lambda}\subset \Kal{\lambda^\prime}$ holds for any 
$ 0\leq \lambda <\lambda^\prime$.  
Furthermore, 
\[
\Kal{\lambda} = \bigcap_{\lambda^\prime >\lambda }\Kal{\lambda^\prime} 
= \overline{ \bigcup_{\lambda^\prime <\lambda }\Kal{\lambda^\prime}}.
\]
Finally, 
since $\lim_{\ell\to\infty}\phi_\beta (\ell )/\ell = 0$,  it is always the case that
\[
 \liminf_{n\to\infty}\frac1n \log A_n (f) \geq 0.
\]
Putting all these observations together, we deduce that, for any $f\in\bbR^d$,
\begin{equation}  
\label{eq:aMGF}
\Lambda_a (f) = \lim_{n\to\infty}\frac1n\log A_n (f) = 
\begin{cases}
\lambda , \ \  &{\rm if}\   f\in \partial \Kal{\lambda},\\
0, &{\rm if}\ f\in \Ka .
\end{cases}
\end{equation}
Similarly, since $0\in {\rm supp}(V)$, 
\begin{equation}  
\label{eq:qMGF}
\Lambda_q (f) = \lim_{n\to\infty}\frac1n\log Q_n (f) =
\begin{cases}
\lambda , \ \  &{\rm if}\   f\in \partial \Kql{\lambda},\\
0, &{\rm if}\ f\in \Kq .
\end{cases}
\end{equation}
Obviously, the distribution of $X(\gamma  )/n$ is exponentially tight under both $\bbA_n^h$ and $\bbQ_n^h$. It follows that the annealed large deviation principle is satisfied with the rate function
\[
\sup_f\lbr f\cdot v - \Lambda_a (h +f )\rbr +\Lambda_a (h ) = 
\sup_f\lbr f\cdot v - \Lambda_a ( f )\rbr + \lb \Lambda_a (h ) -  h\cdot  v\rb .
\]
The latter is easily seen to coincide with $J_a$ in~\eqref{eq:LDFunctions}, using~\eqref{eq:aMGF} and $\fra_{\lambda}(v ) = \max_{f\in\partial\Kal{\lambda}}v\cdot f$. 
The quenched case is dealt with in the same way.\qed

\subsection{Ramifications  for Ballistic Behaviour.} 
The assertion of Theorem~\ref{thm:ballistic} is now straightforward. Set $\fra\equiv \fra_0$ and $\frq\equiv \frq_0$.

\noindent
(1) Since $\fra^* (h) \fra (v) \geq h\cdot v$, we infer that, for any $v\neq 0$,
\begin{equation} 
\label{eq:subspeed}
J_a^h(v)\geq (1-\fra^* (h) )\fra (v ) > c(h) \abs{v} >0 , 
\end{equation}
whenever $h\in {\rm int}\Ka$. The same argument also applies in the quenched case.

Note that formula~\eqref{eq:LDFunctions} readily implies that $J_a^h( 0) =\Lambda_a (h) $, respectively $J_q^h ( 0) = \Lambda_q (h)$. In particular, $J_a^h (0) = 0$ (respectively $J_a^h (0) = 0$) whenever $h\in\partial\Ka$ (respectively $h\in\partial\Kq$).

\noindent
On the other hand, in the ballistic case of super-critical drifts $h\in \partial\Kal{\lambda}$ or, respectively, $h\in \partial\Kql{\lambda}$, for some $\lambda >0$, the value of the corresponding rate functions at zero is strictly positive (and is equal to $\lambda$).

\noindent
(2) As we shall explain in more details in Subsection~\ref{sub:Geometry} and in Subsection~\ref{ssec:OZ}, in the annealed case the control is complete: Outside $\Ka$ the function $\Lambda_a (\cdot )$ is locally analytic and ${\rm Hess}[\Lambda_a ]$ is non-degenerate. Consequently there is a unique minimum $v = \nabla\Lambda_a (h)$ of $J_a^h$  for any super-critical $h\not\in\Ka$. 

\noindent
(3) Following Flury~\cite{Flury-ECP}, zeroes of the quenched rate function can be described as follows.
\begin{lemma}
\label{lem:MhSet}
Let $\mu > 0$ and  $h\in \partial \Kql{\mu}$. Then the set $\calM_h \df \setof{v}{J_q^h (v) = 0}$ in~\eqref{eq:ballisticOutq} can be characterized as follows:
\begin{equation}  
\label{eq:MhSet}
v\in\calM_h\Longleftrightarrow 
\begin{cases}
\frq_\mu (v) = h\cdot v ,\\
\frac{\dd^-}{\dd \lambda}\big|_{\lambda =\mu} \frq_\lambda (v) 
\geq 1\geq \frac{\dd^+ }{\dd \lambda}\big|_{\lambda =\mu} \frq_\lambda (v) .
\end{cases}
\end{equation}
In particular, $\calM_h =\lbr v\rbr$ is a singleton if and only if $\partial \Kql{\mu}$ is smooth at $h$ and $\frq_\lambda (v )$ is smooth at $\mu$. 
\end{lemma}
\begin{proof}
By~\eqref{eq:LDFunctions}, 
\[
v\in\calM_h\Longleftrightarrow  
\max_{\lambda} \lbr \frq_\lambda - \lambda\rbr +\lb \mu - h\cdot v\rb = 0.
\]
The choice $\lambda =\mu$ implies that $\frq_\mu (v) \leq h\cdot v$. Since $h\in\partial
\Kql{\mu}$, the first condition in the rhs of~\eqref{eq:MhSet} follows. Consequently, for any $\lambda$, 
\[
\frq_\lambda (v ) - \frq_\mu (v) \leq \lambda -\mu .
\]
Since $\frq_\lambda$ is concave in $\lambda$, both right- and left-derivatives are defined and the second condition in the rhs of~\eqref{eq:MhSet} follows as well.
\qed\end{proof}

As will be explained in   Subsection~\ref{sec:weak}, the existence of a unique minimizer $v=\nabla\Lambda_q (h)=\nabla\Lambda_a (h)$ of the quenched rate function easily follows from the corresponding annealed statement in the weak disorder regime. Moreover, an almost-sure CLT can be established when the disorder is very weak; see Section~\ref{sec:veryweak}.

\subsection{Proof of Lemma~\ref{lem:series}}
The annealed case is easy. Since the potential is attractive, the Lyapunov exponent $A_\lambda $ is super-additive. Hence, the second limit in~\eqref{eq:LE} is well-defined and, in addition, 
\begin{equation}  
\label{eq:Asuper}
A_\lambda (x) \leq e^{-\fra_\lambda (x)} .
\end{equation}
Since $h\cdot x\leq \fra_\lambda^* (h)\fra_\lambda (x)$, the bound~\eqref{eq:hin} follows from~\eqref{eq:Asuper} and~\eqref{eq:LEproperties} and holds for all sub-critical drifts $h\in {\rm int}\Kal{\lambda}$.  

In the super-critical case $\fra_\lambda^* (h) >1$, pick a unit vector $x$ satisfying $h\cdot x = \fra_\lambda^* (h)\fra_\lambda (x)$.  Then,
\[
A_\lambda ( mx ) e^{m h\cdot x}
\geq \exp\Bigl\{ \frac{m (\fra_\lambda^* (h) - 1)\fra_\lambda (x)}{2}\Bigr\} ,
\]
for all $m$ sufficiently large. Obviously, only paths with  $\abs{\gamma}\geq m$ can contribute to $A_\lambda ( mx )$. On the other hand, for any $\lambda >0$ one can ignore paths with $\abs{\gamma}\geq c_\lambda m$ for some $c_\lambda$ sufficiently large.  It follows that one can find $\alpha >0$, $n_0 >0$ and $y_0$ such that
\[
e^{h\cdot y_0} A_{\lambda , n_0} (y_0 ) \geq e^{2\alpha n_0} .
\]
In view of subadditivity, the target~\eqref{eq:hout} follows by setting $n = kn_0 + r$ and iterating.

\smallskip
The quenched case is slightly more involved. Under suitable assumptions on $V$ (e.g.\ boundedness of ${\rm supp}(V)$ or $\bbE V <\infty $ ), the existence of 
\begin{equation}
\label{eq:QLEbound}
\frq_\lambda (x)  =  -\lim_{N\to\infty} \frac1N \log Q_\lambda  (\lfloor Nx\rfloor  )
=  -\lim_{N\to\infty} \frac1N  \bbE \log Q_\lambda  (\lfloor Nx\rfloor  )
\end{equation}
follows from the subadditive ergodic theorem~\cite{Zerner, KMZ}.

In order to mimic the proofs of~\eqref{eq:hin} and~\eqref{eq:hout}, one needs to apply concentration inequalities in order to control fluctuations of the random quantities on the lhs of~\eqref{eq:QLEbound} around their expectations. This is done in~\cite{Zerner},  under the assumption of $\bbE V^d <\infty$. The speed of convergence of the expectations on the rhs of~\eqref{eq:QLEbound} is under control exactly as in the annealed case. 
\qed

\section{Geometry of Typical Polymers}

\subsection{Skeletons  of Paths.}
Let $\lambda >0$ and $x\in\bbZ^d$ be a distant point. Our characterization of the path measures $\bbQ_x^\lambda$ and $\bbA_x^\lambda$ hinges upon a renormalization construction. In the sequel, $U_\lambda$ denotes the unit ball in either the quenched ($\frq_\lambda$) or the annealed ($\fra_\lambda$) norms. We choose a large scale $K$ and use the dilated shifted balls $KU_\lambda(u) \df u+KU_\lambda$ for a coarse-grained decomposition of paths $\gamma\in\calD ( x)  \df \lbr \gamma :0\mapsto x\rbr$
(see Section~2.2 of  \cite{IoffeVelenik-Annealed}), 
\begin{equation}  
\label{eq:pathDecomp}
\gamma = \gamma_1\cup\eta_1\cup\gamma_2\cup\ldots \cup\eta_m\cup \gamma_{m+1} .
\end{equation}
This decomposition enjoys properties~(a)-(d) below: 

(a) For $i=1, \ldots, m$, the paths $\gamma_i$ are of the form $\gamma_i  : u_{i-1} \mapsto v_i\in\partial KU_\lambda (u_{i-1 })$.

(b) The last path $\gamma_{m+1} :u_m\mapsto x$.

\smallskip
\noindent
Given a set $G$, let us say that a path $\gamma$ with endpoints $u$ and $v$ is in $\calD_G (u ;v )$ if $\gamma\setminus v\subseteq G$. Define $G_1 = KU_\lambda$ and $G_i = KU_\lambda (u_{i-1})\setminus \lb\cup_{j<i} KU_\lambda (u_{j-1}) \rb $ .

\smallskip
(c) $\gamma_i \in\calD_{G_i} (u_{i-1} , v_i )$. 

(d) The paths $\eta_i$ are of the form $\eta_i : v_i\mapsto u_i$ and $\eta_i\cap G_{j-1} = \emptyset$ for any $j <i$. 

\smallskip
\noindent
\textbf{Definition.} 
The set $\hat\gamma_K \df \lb 0=u_0, v_1, u_1, \dots , v_m, u_m =x\rb$ is called the $K$-skeleton of $\gamma$. We say that $\ungamma = (\gamma_1, \dots ,\gamma_{m+1})\sim \hat\gamma_K$ and $\ueta = (\eta_1 , \dots , \eta_m )\sim\hat\gamma_K$ if they satisfy Conditions~(a)-(d) above. 

\smallskip 
\noindent
The collection $\ueta$ is called the hairs of $\hat\gamma_K$. In the sequel we shall concentrate on controlling the geometry of the skeletons. The geometry of hairs is, for every $\lambda >0$ fixed, controlled by a crude comparison with killed random walks, and we refer to Section~2.2 of \cite{IoffeVelenik-Annealed} for the corresponding arguments. 

\smallskip 
\noindent
It follows from Condition~(c) that the paths $\gamma_i$ are pair-wise disjoint. Consequently,
\begin{itemize}
\item In the annealed case, $\Phi_\beta (\gamma_1 \cup\dots \cup\gamma_{m+1}) = \sum_i \Phi_\beta (\gamma_i )$. As a result,
\begin{equation}  
\label{eq:Adecouple}
\sum_{\ungamma\sim\hat\gamma_K} a_\lambda (\gamma_1 \cup\dots \cup\gamma_{m+1})  = \prod 
A_\lambda (u_{i-1 } ; v_i\big| G_i ) ,
\end{equation}
with the obvious notation $A_\lambda (u; v \big|G) \df \sum_{\gamma\in \calD_G (u ;v )} a_\lambda (\gamma )$. 
\item In the quenched case,
\begin{equation}  
\label{eq:Qdecouple}
\sum_{\ungamma\sim\hat\gamma_K} q_\lambda (\gamma_1 \cup\dots \cup\gamma_{m+1})  = \prod Q_\lambda (u_{i-1 } ; v_i\big| G_i ) ,
\end{equation}
and the variables $Q_\lambda (u_{i-1}; v_i \big|G_i) \df \sum_{\gamma\in \calD_{G_i (u_{i-1} ;v_i )}} q_\lambda (\gamma )$ are jointly independent.
\end{itemize}

\subsection{Annealed models.}
\label{ssec:annealed}
Let $\lambda>0$ and $h\in\partial\Kal{\lambda}$ such that $h\cdot x = \fra_\lambda(x)$. Observe first that, by the very definition of $\fra_\lambda$,
\begin{equation}
A_\lambda(x) \geq e^{-\fra_\lambda(x)(1+o(1))}.
\end{equation}
Now, let $\hat\gamma_K = (u_0=0, v_1, u_1, \ldots , v_m, u_m=x)$ be a $K$-skeleton. On the one hand, \eqref{eq:Adecouple} and~\eqref{eq:Asuper} imply that
\[
\log \sum_{\ungamma\sim\hat\gamma_K} a_\lambda (\ungamma) \leq -Km.
\]
On the other hand, independently of the scale $K$,
\begin{equation}
\log\sum_{\ueta\sim\hat\gamma_K} e^{-\lambda\abs{\ueta}} P (\ueta) \leq c_1 (\lambda )m.
\end{equation}
Notice that $\fra_\lambda(u_i-u_{i-1})= K + O(1)$ by construction. We deduce that
\begin{align*}
\bbA_x^\lambda (\hat\gamma_K)
&\leq \exp\Bigl\{ - (1-\epsilon_K)\sum_{i=1}^m \fra_\lambda(u_i-u_{i-1}) - \fra_\lambda(x) 
+\smo{1}\abs{x}
\Bigr\}\\
&\leq \exp\Bigl\{ - (1-\epsilon_K)\sum_{i=1}^m \bigl(\fra_\lambda(u_i-u_{i-1}) - h\cdot(u_i-u_{i-1}) \bigr) 
+\smo{1}\abs{x}
\Bigr\}\\
&= \exp\Bigl\{ - (1-\epsilon_K) \sum_{i=1}^m \frs_a^h(u_i-u_{i-1})
+\smo{1}\abs{x}
\Bigr\},
\end{align*}
where we have introduced the (annealed) surcharge function $\frs_a^h(y) \df \fra_\lambda(y) - h\cdot y$, and $\epsilon_K$ can be chosen arbitrarily small, provided that $K$ is chosen large enough. Defining the (annealed) surcharge of a skeleton $\hat\gamma_K$ by
\[
\frs_a^h(\hat\gamma_K) \df \sum_{i=1}^m \frs_a^h(u_i-u_{i-1}),
\]
we finally obtain the following fundamental surcharge inequality (see~\cite{IoffeVelenik-Annealed} for details):
\begin{lemma}
\label{lem:surcharge-annealed}
For every small $\epsilon>0$, there exists $K_0(d,\beta,\lambda,\epsilon)$ such that
\[
\bbA_x^\lambda \bigl( \frs_a^h(\hat\gamma_K) > 2\epsilon|x| \bigr) \leq e^{-\epsilon |x|},
\]
uniformly in $x\in\Zd$, $h\in\partial\Kal{\lambda}$ such that $h\cdot x = \fra_\lambda(x)$, and scales $K>K_0$.
\end{lemma}

\subsection{Quenched Models.}
In the quenched case, (logarithms of) partition functions are random quantities and we need to control both the averages and the fluctuations.
\begin{lemma}
\label{lem:QxConc}
For any $\lambda > 0 $, there exists $c = c(\lambda ) >0 $ such that
\begin{equation}  
\label{eq:QxConc}
\bbP\bigl(\bigl| \log Q_\lambda (x) - \bbE \log Q_\lambda (x)\bigr|  \geq t\bigr) \leq 
\exp\bigl\{ -c \frac{t^2}{\abs{x}}\bigr\}  ,
\end{equation}
uniformly in $\abs{x}$ large enough.  
\end{lemma}
\begin{proof}
We follow Flury~\cite{Flury-LE}, although working with the conjugate $\lambda$-ensemble helps. For a given realization $\underline{\frv} = \lbr{ \frv_x}\rbr$ of the environment, define
\[
 F_\lambda^x [\underline{\frv} ] \df \log \sum_{X (\gamma ) = x} e^{-\lambda \abs{\gamma}
-\sum_y \frv_y \ell_\gamma (y)} P(\gamma)\ \ 
{\rm and}\ \ 
\bbQ_\lambda^{x , \underline{\frv}}  (\gamma ) \df \frac{e^{-\lambda \abs{\gamma}
-\sum_y \frv_y \ell_\gamma (y)} P(\gamma)}{e^{F_\lambda^x [\underline{\frv} ]} }. 
\]
Since $\lambda\geq 0$ and the entries of $\underline{\frv}$ are non-negative,  there exists $c = c(\lambda ) $ such that
\begin{equation}  
\label{eq:gammaBound}
\bbQ_\lambda^{x, \underline{\frv}} \Bigl( \sum_z \ell_\gamma^2 (z)  \Bigr)   \leq c \abs{x} .
\end{equation}
In order to see this, given $z\in\bbZ^d$, define the set of loops
\[
 \calL_z \df \setof{\eta :z\mapsto z}{\ell_\eta  ( z ) = 1} .
\]
Evidently,
\[
 \sum_{\eta \in \calL_z}Q_\lambda (\eta )\leq e^{-\lambda} .
\]
Now, any path $\gamma\in\calD (x)  $ with $\ell_\gamma (z) = n$ has a well-defined decomposition
\[
\gamma = \gamma_0\cup\eta_1\cup\dots\cup\eta_{n-1},
\]
with $\ell_{\gamma_0} (z) = 1$ and $\eta_i\in\calL_z$.  It follows that
\begin{equation}  
\label{eq:lzBound} 
\bbQ_\lambda^{x ,\underline{\frv}}  \lb \ell_\gamma (z)^2\big| \ell_z (\gamma )>0\rb \leq 
\sum_n n^2 e^{-\lambda (n-1)} \df c_1 (\lambda ) , 
\end{equation}
uniformly in the realizations $\underline{\frv}$ of the environment. Consequently, 
\[
\bbQ_\lambda^{x ,  \underline{\frv}} \Bigl( \sum_z \ell_\gamma^2 (z)  \Bigr)  \leq 
c_1(\lambda )\sum_z \bbQ_ \lambda^{x , \underline{\frv} }\bigl( \ell_\gamma (z ) >0\bigr) 
\leq c_1 \bbQ_\lambda^{x ,  \underline{\frv}} \lb \abs{\gamma}\rb \leq c_2 \abs{x} .
\]
The last inequalllity above is straightforward since we assume that 
$\lambda >0$ and that the distribution of random environment has bounded support.

At this stage, we infer that $F_\lambda^x [\cdot ]$ is Lipschitz: Given two realizations of the environment 
$\underline{\frw}$ and $\underline{\frv}$, define $\frv^t_\cdot \df t\frw_\cdot + (1-t)\frv_\cdot$. Then, 
\[
F_\lambda^x [ \underline{\frw} ] - F_\lambda^x [ \underline{\frv} ] = \int_0^1 
\bbQ_\lambda^{x , \underline{\frv}^t}\Bigl(
\sum_z (\frw_z  - \frv_z )\ell_\gamma (z) \Bigr) \dd t \leq 
\sqrt{c\abs{x} }\cdot 
\|\underline{\frw} -\underline{\frv}  \|_2 .
\]
Indeed, 
\[
\bbQ_\lambda^{x , \underline{\frv} }\lb  \sqrt{\sum_z\ell_\gamma (z)^2}\rb \leq 
\sqrt{  \bbQ_\lambda^{x  , \underline{\frv} } \lb  \sum_z \ell_\gamma (z )^2\rb } 
\]
and~\eqref{eq:gammaBound} applies. Since $F_\lambda^x [\cdot ]$ is convex
and, as we have checked above, Lipschitz, \eqref{eq:QxConc} follows from concentration inequalities on product spaces (see, e.g., \cite[Corollary~4.10]{Ledoux}) .  
\qed\end{proof}

Lemma~\ref{lem:QxConc} leads to a lower bound on the random partition function $Q_\lambda (x)$. Define
\[
\hat \epsilon (r) \df - \min_{ \frq_\lambda (z )\geq r} \lbr \frac{ \frq_\lambda (z ) + 
 \bbE\log Q_\lambda (z) }{\frq_\lambda (z )}\rbr .
\]
By subadditivity, $\hat \epsilon (r)$ is non-negative, and  $\lim_{r\to\infty} \hat \epsilon (r ) = 0$. By~\eqref{eq:QxConc},
\begin{equation}  
\label{eq:QlControl}
\bbP\lb Q_\lambda (x ) \leq e^{- \frq_\lambda (x)\lb 1+\hat \epsilon (\abs{x}) +t\rb }\rb 
\leq \exp\lbr -c  {t^2}{\abs{x}}\rbr .
\end{equation}
We may thus assume that there exists $\epsilon (r)\to 0$ such that
\begin{equation}
\label{eq:qlBound}
\log  Q_\lambda (x ) \geq -  \frq_\lambda (x) \lb 1+\epsilon (\abs{x})\rb  ,
\end{equation}
$\bbP$-a.s.\ for all $\abs{x}$ sufficiently large.

\smallskip
The lower bound~\eqref{eq:qlBound} is used to control the geometry of the skeletons $\hat\gamma_K$. Namely, 
\begin{equation}  
\label{eq:logQl}
\log Q_\lambda (\hat\gamma_K ) = \log\sum_{\ueta\sim\hat\gamma_K} 
q_\lambda (\ueta ) + \sum_{i=1}^m \log Q_\lambda\lb u_{i- 1} ; v_i\big| G_i \rb . 
\end{equation}
A comparison with the simple random walk killed at the constant rate $\lambda >0$ reveals that the following bounds hold uniformly in the realizations of the environment:
\[
\log\sum_{\ueta\sim\hat\gamma_K}q_\lambda (\ueta ) \leq c_2 (\lambda )m\ \ {\rm and}\ \ 
\log Q_\lambda\lb u_{i- 1} ; v_i\big| G_i \rb  \leq  -c_3 (\lambda )K .
\]
It follows that we may restrict our attention to {\em moderate} trunks with at most $m\leq c_4\abs{x}/K$ vertices.  Consequently, the first term in~\eqref{eq:logQl} is at most of order $c_2 c_4\abs{x}/K$.  

Assuming that $m\leq c_4\abs{x}/K$, let us focus on the second term in~\eqref{eq:logQl}. To simplify notations, we shall describe it as a random variable $F_{\hat\gamma_K} \df \sum_i\log Q_\lambda\lb u_{i- 1} ; v_i\big| G_i \rb$. First of all, since both $u_i$ and $v_i$ belong to $\partial K U_\lambda (u_{i-1} )$, 
\[
\bbE F_{\hat\gamma_K}  \leq -\sum \frq_\lambda (v_i -u_{i-1} ) = 
- \sum \frq_\lambda (u_i -u_{i-1}) + O\lb\frac{\abs{x}}{K}\rb .
\]
\begin{lemma}
\label{lem:FConc}
For any $\lambda > 0 $, there exists $c = c(\lambda ) >0 $ such that
\begin{equation}  
\label{eq:FConc}
\bbP\bigl(\bigl| F_{\hat\gamma_K} - \bbE F_{\hat\gamma_K} \bigr|  \geq t\bigr) \leq 
{\exp}\bigl\{-\frac{t^2}{\abs{x}}\bigr\}  ,
\end{equation}
uniformly in $\abs{x}$ large enough, in renormalization scales $K$ and in moderate skeletons $\hat\gamma_K$.
\end{lemma}
The proof of this lemma is similar to the proof of Lemma~\ref{lem:QxConc} and we shall sketch it below. The size of the scale $K$ is not essential for the proof. It is essential, however, for an efficient use of the lemma: Assuming that~\eqref{eq:FConc} holds, we choose $1\gg \delta \gg\sqrt{ \log K/K} $. By~\eqref{eq:FConc},
\[
\bbP\bigl(\bigl| F_{\hat\gamma_K} - \bbE F_{\hat\gamma_K} \bigr|  \geq \delta \abs{x}\bigr) \leq 
\exp\lbr  -   c_5 \delta^2 \abs{x} \rbr , 
\]
for any moderate trunk $\hat\gamma_K$. Since there are at most $\exp\lbr c_6 \frac{\log K}{K}\abs{x}\rbr$ moderate trunks, we conclude that
\begin{lemma}
\label{lem:SkBound}
For any $\delta >0$, there exists a finite scale $K$ such that
\begin{equation}  
\label{eq:SkBound}
F_{\hat\gamma_K} \leq  -\sum_i \frq_\lambda (u_i - u_{i-1} ) +\delta \abs{x} ,
\end{equation}
$\bbP$-a.s.\ for all $\abs{x}$ large enough (and all the corresponding moderate skeleton trunks of paths $\gamma \in\calD ( x )$).
\end{lemma}
\begin{proof}[of Lemma~\ref{lem:FConc}]
Introduce the following notation: Given a realization $\underline{\frv}_i$ of the environment on $G_i$, let $\bbQ_\lambda^{\ufrv_i } \lb\cdot\big| G_i \rb $ be the corresponding probability distribution on the set of paths $\calD_{G_i } (u_{i-1} , v_{i})$. In this notation,
\[
F_{\hat\gamma_K} (\ufrw ) - 
F_{\hat\gamma_K} (\ufrv ) = \sum_{i=1}^{m}
\int_0^1 \bbQ_\lambda^{\ufrv_i^t } \Bigl( \sum_{z\in G_i }
\ell_{\gamma_i }(z ) \lb \frw_z  -\frv_z \rb 
\Bigm| G_i \Bigr)  ,
\]
where $\ufrv^t = t\ufrw + (1-t )\ufrv $.  The conclusion follows as in the proof of Lemma~\ref{lem:QxConc}.
\qed
\end{proof}
We can now proceed as in the annealed case and introduce the (quenched) surcharge of a skeleton $\hat\gamma_K$,
\[
\frs_q^h(\hat\gamma_K) \df \sum_{i=1}^m \bigl( \frq_\lambda(u_i-u_{i-1}) - h\cdot (u_i-u_{i-1}) \bigr).
\]
We then obtain the following quenched version of the surcharge inequality:
\begin{lemma}
\label{lem:surcharge-quenched}
For every small $\epsilon>0$, there exists $K_0(d,\beta,\lambda,\epsilon)$ such that, $\bbP$-a.s.,
\[
\bbQ_x^\lambda \bigl( \frs_q^h(\hat\gamma_K) > 2\epsilon|x| \bigr) \leq e^{-\epsilon |x|},
\]
uniformly in sufficiently large $x\in\Zd$, $h\in\partial\Kql{\lambda}$ such that $h\cdot x = \frq_\lambda(x)$, and scales $K>K_0$.
\end{lemma}

\subsection{Irreducible decomposition and effective directed structure.}

The surcharge inequalities of Lemmas~\ref{lem:surcharge-annealed} and~\ref{lem:surcharge-quenched} pave the way to a detailed analysis of the structure of typical paths, as they reduce probabilistic estimates to purely geometric ones. We only describe here the resulting picture, but details can be found in~\cite{IoffeVelenik-Annealed}. 

\smallskip
Let $\lambda >0$ and $h\in\partial\Kal{\lambda}$. Let us fix $\delta\in(0,1)$. We define the forward cone by
\[
\fcone = \setof{y\in\Zd}{\frs_a^h(y) < \delta\fra_\lambda(y)},
\]
and the \emph{backward cone} by $\bcone = -\fcone$. Given $\gamma=(\gamma_0,\ldots,\gamma_n):0\to x$, we say that $\gamma_k$ is a cone-point of $\gamma$ if
\[
\gamma \subset \bigl( \gamma_k + \fcone \bigr) 
\cup \bigl( \gamma_k + \bcone \bigr).
\]
The next theorem shows that typical paths have a positive density of cone-points.
\begin{theorem}\cite{IoffeVelenik-Annealed}
\label{thm:ConePoints}
Let $\#_{\scriptscriptstyle\rm cone}(\gamma)$ be the number of cone-points of $\gamma$. There exist $c,C>0$ and $\delta^\prime >0$, depending only on $d,\beta,\delta$ and $\lambda$, such that
\begin{equation}  
\label{eq:AlBound}
\bbA_\lambda^x \bigl( \#_{\scriptscriptstyle\rm cone}(\gamma) < c |x| ) \leq e^{-C |x|},
\end{equation}
uniformly in all sufficiently large $x\in\Zd$ satisfying  $\frs_a^h(x) \leq \delta^\prime \fra_\lambda(x)$.
\end{theorem}
\begin{remark}
By \eqref{eq:aMGF} and Theorem~\ref{thm:LD},
there exist $\alpha\in [1, \infty )$ and $c^\prime  > 0$ such that
\begin{equation}  
\label{eq:AhDecomp} 
1 \asymp e^{- \lambda n} A_n (h) = 
\sumtwo{\alpha^{-1}n \leq |x| \leq \alpha n}{\frs_a^h(x) \leq \delta^\prime \fra_\lambda(x)}
 A_{\lambda ,n} (x)e^{h\cdot x} + 
\smo{e^{-c^\prime n }} , 
\end{equation}
as $n$ becomes large. It follows that  sets of paths which are uniformly exponentially improbable under the $\bbA^x_\lambda$-measures will remain exponentially improbable under $\bbA_n^h$.  In particular, \eqref{eq:AlBound} implies that there exist $c,C>0$, depending only on $d,\beta,\delta$ and $h$ such that
\begin{equation}  
\label{eq:AhBound}
\bbA_n^h \bigl( \#_{\scriptscriptstyle\rm cone}(\gamma) < c n ) \leq e^{-C n},
\end{equation}
uniformly in $n$ sufficiently large.
\end{remark}

With the help of~\eqref{eq:AhBound}, we can decompose typical ballistic paths into a string of irreducible pieces. A path $\gamma=(\gamma_0,\ldots,\gamma_n)$ is said to be backward irreducible if $\gamma_n$ is the only cone-point of $\gamma$. Similarly, $\gamma$ is said to be forward irreducible  if $\gamma_0$ is the only cone-point of $\gamma$. Finally, $\gamma$ is said to be irreducible if $\gamma_0$ and $\gamma_n$ are the only cone-points of $\gamma$. We denote by $\calF^>(y)$, $\calF^<(y)$ and $\calF(y)$ the corresponding sets of irreducible paths connecting $0$ to $y$.

In view of~\eqref{eq:AhBound}, we can restrict our attention to paths $\gamma$ possessing at least $c^\prime n$ cone-points, at least when $n$ is sufficiently large. We can then unambiguously decompose $\gamma$ into irreducible sub-paths:
\begin{equation}
\label{eq:DecompositionIntoIrreduciblePieces}
\gamma = \omega_{\rm >} \cup \omega_1 \cup \cdots \cup \omega_m \cup \omega_{\rm <}.
\end{equation}
We thus have the following expression
\begin{equation}
\label{eq:DecompositionZx}
e^{- \lambda n} A_n (h )
= 
\sum_{m\geq c^\prime n}\, \sum_{\omega_{>}\in\calF^>}
\sum_{\omega_1, \dots \omega_m\in\calF}
\sum_{\omega_<\in\calF^<}
a_{\lambda,h}(\gamma)
\1_{\{\abs{\gamma}=n\}}
+
O\bigl( e^{- Cn } \bigr) \,.
\end{equation}
Observe now that the weight $a_{\lambda,h}(\gamma)$ of a path $\gamma$ can be nicely factorized over its irreducible components (see~\eqref{eq:DecompositionIntoIrreduciblePieces}):
\[
a_{\lambda,h}(\gamma) = 
a_{\lambda,h}(\omega_>)\, 
a_{\lambda,h}(\omega_<)\, \prod_{i=1}^m 
a_{\lambda,h}(\omega_i),
\]

\smallskip
\noindent
Similarly, Lemma~\ref{lem:surcharge-quenched} implies:
\begin{theorem}
\label{thm:QConePoints}
Let $\lambda > 0$ and $h\in\partial\Kql{\lambda}$. 
There exist $c,C>0$, depending only on $d,\beta,\delta$ and $h$ such that
\begin{equation}  
\label{eq:QhBound}
\bbQ_n^h \bigl( \#_{\scriptscriptstyle\rm cone}(\gamma) < c n ) \leq e^{-C n},
\end{equation}
$\bbP$-a.s., uniformly in $n$ sufficiently large. In particular, using the same notation~\eqref{eq:DecompositionIntoIrreduciblePieces} for the irreducible decomposition of $\gamma$, 
\begin{equation}
\label{eq:DecompositionZQx}
e^{- \lambda n} Q_n (h )
= 
\sum_{m\geq c^\prime n}\, 
\sum_{\omega_{>}\in\calF^>}
\sum_{\omega_1, \dots \omega_m\in\calF}
\sum_{\omega_<\in\calF^<}
q_{\lambda,h}(\gamma)
\1_{\{\abs{\gamma}=n\}}
+
O\bigl( e^{- Cn } \bigr) \,.
\end{equation}
$\bbP$-a.s., for all $n$ large enough.
\end{theorem}

\subsection{Basic Partition Functions.} 
Let us say that a path $\gamma:0\to x$ is cone-confined if $\gamma\subset\fcone\cap(x+\bcone)$.
Let $\calT (x)\subset \calD (x)$ be the collection of all cone-confined paths leading from $0$ to $x$, and let $\calF (x)\subset\calT (x)$ be the collection of all irreducible cone-confined paths. Clearly $\omega_{>} = \omega_{<} =\emptyset$ in the irreducible decomposition~\eqref{eq:DecompositionIntoIrreduciblePieces} of paths $\gamma\in \calT (x)$. Let $\lambda >0$ and $h\in\partial \Kal{\lambda}$. We define the following in general unnormalized quenched partition functions,
\begin{equation}  
\label{eq:weights}
\tq{x ,n} \df \sum_{\gamma\in\calT (x)}\1_{\lbr \abs{\gamma} =n\rbr} q_{\lambda ,h}
\ \ {\rm and}\ \ 
\fq{x ,n} \df \sum_{\gamma\in\calF (x)}\1_{\lbr \abs{\gamma} =n\rbr} q_{\lambda ,h} .
\end{equation}
Their annealed counterparts are denoted by $\ta{x,n} \df \bbE\tq{x, n}$ and $\fa{x,n } \df \bbE\fq{x, n}$. 
As we shall see below, $\lbr \fa{x ,n}\rbr$ is normalized - it is a probability distribution,
\[
\sum_{n}\sum_{x}\fa{x, n} \df \sum_n \fa{n} = 1 .
\]
with exponentially decaying tails. The tails are exponential by Theorem~\ref{thm:ConePoints}. The probabilistic
normalization is explained in Subsection~\ref{sub:renewal}. 

For $n\geq 1$, let $\ta{n}\df\sum_x\ta{x,n}$ and $\fa{n}\df\sum_x\fa{x,n}$, and set $\ta{0}\df 1$. The irreducible decomposition~\eqref{eq:DecompositionIntoIrreduciblePieces} of paths imply the following renewal-type relations for $\ta{n}$ and for $\tq{x,n}$:
\begin{equation}
\label{eq:renewal}
\ta{n} = \sum_{m=0}^{n-1} \ta{m}\fa{n-m},\qquad
\tq{x,n} = \sum_{m=0}^{n-1}\sum_y \tq{y,m} f^{\theta_y\omega}_{x-y,n-m}.
\end{equation}
For the rest of the paper we shall work mainly with the above  basic ensembles of paths.  All the results can be routinely extended (as in, e.g.,\cite{IoffeVelenik-AOP}) to general ensembles  by summing out over the paths $\omega_{<}$ and $\omega_{>}$ paths in the irreducible decomposition~\eqref{eq:DecompositionIntoIrreduciblePieces}, their weights being exponentially decaying.

\section{The Annealed Model}
\subsection{Asymptotics of $\ta{n} = \sum_x \ta{x, n}$.}
\label{sub:renewal}
Annealed asymptotics are not related to the strength of disorder and hold for all values of $\beta\geq 0$. 
Neither do they require any moment assumptions on $V$. 
\begin{lemma}
\label{lem:tanAsympt}
Let $\lambda >0$ and $h\in\partial\Kal{\lambda}$. There exists $\kappa = \kappa (\lambda ,h )$ such that
\begin{equation}  
\label{eq:tanAsympt}
\lim_{n\to\infty} \ta{n} = \Bigl\{ \sum_n n\fa{n}\Bigr\}^{-1} \df \frac{1}{\kappa}
\end{equation}
exponentially fast. 
\end{lemma}
\begin{proof}
For $\abs{u}\leq 1$, define the generating functions
\[
\hat{\mathbf t} [u] \df \sum_{n=0}^\infty  u^n \ta{n}\ \ {\rm and}\ \ 
\hat{\mathbf f} [u] \df \sum_{n=1}^\infty  u^n \fa{n} .
\]
It follows from Theorem~\ref{thm:ConePoints} that the second series converges on some disc $\bbD_{1+\nu} = \setof{u\in\bbC}{\abs{u} < 1+\nu}$. The first series blows up at any  $\bbR\ni u >1$. 
Since by \eqref{eq:renewal}, $\hat{\mathbf t} [u] = \lb 1- \hat{\mathbf f} [u] \rb ^{-1}$, 
it follows that $\hat{\mathbf f}  [1] =1$, which is the probabilistic normalization mentioned above. We identify $\kappa$ in~\eqref{eq:tanAsympt} as $\kappa = \hat{\mathbf f}^\prime [1]$. It then follows from the renewal relation~\eqref{eq:renewal} that
\begin{equation}  
\label{eq:htaFormula}
\hat{\mathbf t} [u]  = \frac{1}{1 - \hat{\mathbf f} [u] } = 
\frac{1}{\kappa (1-u)} 
+ \frac{\hat{\mathbf f} [u]  -\hat {\mathbf  f } [1] - \kappa (u-1 )}{\kappa 
(1- \hat{\mathbf f} [u])
(1-u)} \df 
\frac{1}{\kappa (1-u)}  +\Delta  [u] .
\end{equation} 
Since the function $\Delta$ is analytic on some disc $\bbD_{1+\nu^\prime}$, the claim follows from Cauchy's formula. 
\qed\end{proof}

\subsection{Geometry of $\Kal{\lambda}$, annealed  LLN and CLT.}
\label{sub:Geometry}
Let $\lambda >0$ and $h\in \partial \Kal{\lambda}$.
In the ballistic phase of the annealed model, the CLT is obtained on the level of a local limit description: Given $z\in\bbC$, let us try to find $\mu = \mu (z)\in \bbC$ such that
\begin{equation}  
\label{eq:muz}
{\mathbf F}(z ,\mu )\df 
\sum_{n, x} e^{-\mu n + z\cdot x}\fa{x ,n} = 1 .
\end{equation}
Since $\lbr \fa{n}\rbr$ has exponential decay, the implicit function theorem implies that
\begin{lemma}
\label{lem:muz}
There exist $\delta, \eta >0$ and an analytic function $\mu$ on $\bbD_{\delta}^d$ such that 
\begin{equation}  
\label{eq:muanalytic}
\setof{(z ,\mu )\in \bbD_\delta^d \times\bbD_\eta}{{\mathbf F} ( z, \mu ) =1} = 
\setof{(z ,\mu )\in \bbD_\delta^d \times\bbD_\eta}{ \mu = \mu (z) } .
\end{equation}
Moreover, ${\rm Hess} [\mu ] (0 )$ is non-degenerate. 
\end{lemma}
If $z$ is real and $|z |$ is sufficiently small, then $\mu (z) = \Lambda_a (h+z ) -\lambda$. 
Indeed, if $|z |$ is small, then the leading contribution to $\sum_n e^{-(\mu +\lambda) n} A_n (h +z )$ 
is still coming from 
\[
\sum_n  \sum_{x } \ta{x ,n}\,  { e}^{- \mu n + z\cdot x}
\]
By \eqref{eq:renewal} and \eqref{eq:muz}, 
 $\mu = \mu (z )$ describes the radius of convergence of the latter series,
whereas  $\mu +\lambda =   \Lambda_a (h+z ) $ descibes the radius of convergence 
of the former one.

 Therefore, $\partial\Kal{\lambda}$ is locally 
given by the level set $\setof{ h+z}{\mu (z ) = 0}$.  
In addition $\Lambda_a$ inherits analyticity and non-degeneracy properties of $\mu$:
\begin{equation} 
\label{eq:muDerivatives}
\nabla \mu (0 ) = \nabla
\Lambda_a  (h) 
\df v 
= v^a (h ,\beta )
 \ \ \text{and define}\ \ {\rm Hess} [\mu ] (0 ) = 
{\rm Hess }[\Lambda_a ] (h) 
\df\Xi^{-1}  .
\end{equation}
Given $z\in\bbD_\delta$ as above, define 
\[
\fa{x, n}(z)  \df \fa{x, n}e^{-\mu (z) n
+z\cdot x}\ \ \text{and, respectively,}\ \  
\ta{x ,n}(z)  \df \ta{x, n}e^{-\mu (z) n
+z\cdot x}. 
\]
Set  $\fa{n} (z) = \sum_x\fa{x,n} (z),\, \ta{n} (z) = \sum_x\ta{x,n} (z)$ 
and 
$\kappa (z)^{-1} =  \sum_n n\fa{n }(z)$. Literally  repeating the derivation of~\eqref{eq:tanAsympt}, we infer that there exists $\alpha >0$ such that, uniformly in $z\in\bar{\bbD}_\delta^d$, 
\begin{equation}  
\label{eq:tanAsymptz}
\left| \ta{n} (z) - \frac1{\kappa (z)}\right| \leq e^{-\alpha n} .
\end{equation}
By Cauchy's formula $\nabla\log \ta{n} (z) = {\rm O}(1)$.  Therefore, 
\begin{equation}
{\rm O}\lb \frac{1}{n}\rb =  \frac1{n}\nabla\log \ta{n} (0) = -\nabla\mu (0 ) + \frac{1}{n}
\sum_{x}x \frac{\ta{x,n}}{\ta{n}}  = -v +\frac{1}{n}\sum_{x}x \frac{\ta{x,n}}{\ta{n}} .
\label{eq:A-LLN}
\end{equation}
\eqref{eq:A-LLN} is an annealed law of large numbers
which, in particular, identifies $v$ as the limiting macroscopic spatial extension. 

Next, the following form of the annealed CLT holds: for any $\alpha\in\bbR^d$,
\begin{equation}  
\label{eq:AnCLT}
\begin{split} 
\frac{\aS{n} (\alpha  )}{\ta{n}} & \df 
\sum_{x} \frac{\ta{x ,n}}{\ta{n}} \exp\bigl\{ i \alpha \cdot \frac{x - nv }{\sqrt{n}}\bigr\}  = 
\frac{\ta{n}\bigl( \frac{i\alpha }{\sqrt{n}}\bigr)}{\ta{n}} \exp\Bigl\{  n\mu \bigl( 
\frac{i\alpha }{\sqrt{n}}\bigr)  - n v\cdot \frac{i\alpha }{\sqrt{n}}\Bigr\}\\  
&= 
\exp\Bigl\{ -\frac12 \Xi^{-1}\,\alpha\cdot\alpha  \Bigr\} \bigl( 1 + O(n^{-1/2 }) \bigr) ,
\end{split} 
\end{equation}
with the second asymptotic equality holding uniformly in $\alpha$ on compact subsets of $\bbR^d$.  

\subsection{Local limit theorem for the annealed polymer.}
\label{ssec:OZ}
In this Subsection we shall explain the local limit picture behind 
\eqref{eq:A-LLN} and \eqref{eq:AnCLT}.  
Recall that
\begin{align}
\ta{x,n} &= \sum_{m=0}^{n-1}\sum_y \ta{y,m} \fa{x-y,n-m}\nonumber\\
&= \sum_{N\geq 1}\, \sum_{
m_1, \dots ,m_N \geq 1
}\, \sum_{y_1,\dots ,y_N\in\bbZ^d} \1_{\{\sum m_i = n,  \sum y_i = x\}} \prod_{i=1}^N \fa{y_i-y_{i-1},m_i},
\label{eq:expandedrenewal}
\end{align}
where we have set, for convenience, $y_0=0$.
As explained above, the weights $\fa{y,m}$ form a probability distribution on $\bbZ^d\times\bbN$,
\[
\sum_{y\in\bbZ^d,\, m\in\bbN} \fa{y,m} = 1,
\]
and decay exponentially both in $y$ and $m$.
Let us consider an i.i.d. sequence of random vectors $(Y_k,M_k)_{k\geq 1}$ whose joint distribution $\Peff$ is given by these weights. Then~\eqref{eq:expandedrenewal} can be expressed as
\[
\ta{x,n} = \sum_{N\geq 1}\Peff\bigl(\sum_{i=1}^N (Y_i,M_i) = (x,n)\bigr).
\]
Consequently, sharp asymptotics for $\ta{x,n}$ readily follow from a local limit analysis of the empirical mean of the i.i.d. random vectors $(Y_k,M_k)_{k\geq 1}$ with exponential tails. In this way, one obtains the following sharp asymptotics for the extension of an annealed polymer, covering all possible deviation scales.
\begin{theorem}\cite{IoffeVelenik-Annealed}
Suppose that $h\not\in\Ka$. Let $v_h=\nabla\Lambda_a(h)$. Then, for some small enough $\epsilon>0$, the rate function $J_h^a$ is real analytic and strictly convex on the ball $B_\epsilon (v_h)\df \setof{u}{|u-v_h|<\epsilon}$ with a non-degenerate quadratic minimum at $v_h$. Moreover, there exists  a strictly positive real analytic function $G$ on $B_\epsilon(v_h)$ such that
\begin{equation}
\label{BeAsympt}
\bbA_n^h \bigl( \frac{X(\gamma)}{n} = u \bigr) =
\frac{G(u)}{\sqrt{n^d}}\, e^{-n J_h^a (u)} \bigl( 1+ o(1) \bigr) ,
\end{equation}
uniformly in $u\in B_\epsilon(v)\cap\tfrac1n\Zd$.
\end{theorem}
\begin{remark}
 We would like to note that a local limit result for a particular instance of the 
annealed model (discrete Wiener
sausage with a fixed non-zero drift at small $\beta$) was obtained in~\cite{Trachsler}. We are grateful to Erwin Bolthausen for sending us a copy of this work.
\end{remark}

\section{Weak disorder}
\label{sec:weak}
In this section, we focus on the super-critical quenched models in the weak disorder regime.
Accordingly, we consider higher dimensional models on $\bbZ^d$ with $d\geq 4$. 
Let us say that the weak disorder holds at $(h ,\beta )$ if there 
exists $\lambda >0$ such that 
$h\in\partial\Kal{\lambda}$ and the disorder is weak in the conjugate ensemble at $(\lambda ,\beta )$, 
that is the Lyapunov exponents coincide;  $\fra_\lambda (\cdot ) \equiv \frq_\lambda (\cdot )$.
In particular, $\lambda = \Lambda_a (h ) =\Lambda_q (h)$.
\subsection{LLN at  super-critical drifts}
\label{ssec:LLN}
An important, albeit elementary, observation is that, in this regime, events of 
exponentially small probability under the annealed measure $\bbA_n^h$ are also exponentially unlikely under the quenched measure $\bbQ_n^h$.
\begin{lemma}
\label{lem:weakdisorder-expdecay}
Assume that weak disorder holds, $\Lambda_a (h ) =\Lambda_q (h) $. Let $E$ be a path event such that $\bbA_n^h(E) \leq e^{-cn}$ for some constant $c>0$ and all $n$ large enough. Then there exists $c'>0$ such that, $\bbP$-a.s.,
\[
\bbQ_n^h(E) \leq e^{-c'n},
\]
for all $n$ large enough.
\end{lemma}
\begin{proof}
We note that, since $\Lambda_a (h ) =\Lambda_q (h) $, it 
follows from Markov's inequality that, for all $n$ large enough,
\begin{multline*}
\bbP\bigl( \bbQ_n^h(E) > e^{-c' n} \bigr) = \bbP \bigl( Q_n(h;E) > Q_n(h) e^{-c' n} \bigr)\\
\leq \bbP \bigl( Q_n(h;E) > A_n(h) e^{-\tfrac12 c' n} \bigr)
\leq \bbA_n^h(E) e^{+\tfrac12 c' n} \leq e^{-(c-\tfrac12 c') n},
\end{multline*}
where $Q_n(h;E)$ denotes the quenched partition function restricted to paths in $E$.
The conclusion now follows from Borel-Cantelli.
\qed\end{proof}
Recall that a pulling force $h\in\partial\Kal{\lambda}$ is called super-critical if $\lambda >0$. 
Using Lemma~\ref{lem:weakdisorder-expdecay}, it is very easy to prove that, in the super-critical weak disorder regime, the quenched model satisfies LLN, and that the polymer has the same limiting macroscopic extension under the quenched and annealed path measures.
\begin{lemma}
Assume that weak disorder holds, $\lambda = \Lambda_a (h )=\Lambda_q (h) >0$.
Let $v=\nabla\Lambda_a (h)$ be the macroscopic extension of the polymer under the annealed path measure. Then, for any $\epsilon>0$,
\[
\lim_{n\to\infty} \bbQ^h_n \bigl( \bigl|\frac{X(\gamma )}{n} - v \bigr| > \epsilon\bigr) = 0,\qquad \bbP \text{-a.s.,}
\]
exponentially fast in $n$.
\end{lemma}
\begin{proof}
The claim immediately follows from~\eqref{eq:ballisticOutA} and Lemma~\ref{lem:weakdisorder-expdecay}.
\qed\end{proof}

\color{black}

\subsection{Very Weak Disorder}
\label{sec:veryweak} 
Recall that we use notation $\lb \Omega , \calF, \bbP\rb$ for the (product ) probability
space generated by the random environment. 
For the rest of this section, we consider the regime of very weak disorder, which 
should be understood in the following sense: we fix either $\lambda >0$ or $h\neq 0$ and then, for $\beta$ sufficiently small, we pick the remaining parameter ($h$ or $\lambda$) according to $h\in\partial\Kql{\lambda}$. 

\noindent
The regime of very weak disorder is quantified  in terms of the
 following upper bound : 
\begin{lemma}
Fix an external force $h\neq 0$. Then, for all $\beta$ small enough, 
the random weights~\eqref{eq:weights} (with $\lambda = \lambda ( h ,\beta )$ being determined by $h\in\partial \Kal {\lambda} $) satisfy: There exist $c_1,c_2 <\infty$ such that, uniformly in $x_1, x_2 , m_1 , m_2, \ell $ and in sub-$\sigma$-algebras $\calF_1, \calF_2$ of 
$\calF$,
\begin{equation}  
\label{eq:WDBound}
\begin{split}
&\abs{\bbE \tq{x_1 , \ell}\tq{x_2 ,\ell} \bbE\bigl( f^{\theta_{x_1}\omega}_{m_1} -\fa{m_1}\bigm| 
\calF_1 \bigr) \bbE \bigl( f^{\theta_{x_2}\omega}_{m_2} -\fa{m_2} \bigm|\calF^\prime \bigr)} \\
&\quad 
\leq \frac{c_1e^{- c_2 (m_1+m_2 )}}{ \ell^{d}} \exp\Bigl\{ -c_2\bigl( \abs{x_1 - x_2} 
+ \frac{\abs{x_1 -\ell v}^2}{\ell} 
\bigr)\Bigr\} , 
\end{split}
\end{equation}
where $v = \nabla \lambda (h)$. 
\end{lemma}
Although~\eqref{eq:WDBound} looks technical, it has a transparent intuitive meaning: the expressions on the rhs are just local limit bounds for a couple of independent annealed polymers  with exponential penalty for disagreement at  their end-points. In the regime of very weak disorder, the interaction between polymers does not destroy these asymptotics. The proof will be given elsewhere~\cite{IoffeVelenik-inprogress}. Closely related upper bounds were already derived in~\cite{IoffeVelenik-AOP}.

\subsection{Convergence of Partition Functions.}
\label{ssec:Convergence}
As mentioned above, the rescaled quenched partition functions satisfy the following multidimensional renewal relation:  
\begin{equation}  
\label{eq:htqFormula}
\tq{z, n} = \sum_{m=0}^{n-1}\sum_x \tq{x ,m} f^{\theta_x\omega}_{z -x ,n-m} 
\ \ {\rm and}\ \ \tq{ n} = \sum_z \tq{z, n} .
\end{equation}
\begin{theorem}
\label{thm:tqnAsympt} 
In the regime of very weak disorder, 
\begin{equation} 
\label{eq:tqnAsympt}
\lim_{n\to\infty} \tq{n} = \frac{1}{\kappa} \Bigl( 1 + \sum_{x ,y} \tq{x} 
\lb f^{\theta_x\omega}_{y-x}  -\fa{y-x}\rb\Bigr)
\df \frac{1}{\kappa} s^\omega 
 \in (0, \infty ) ,
\end{equation}
$\bbP$-a.s.\ and in $L_2 (\Omega )$. 
\end{theorem}
\begin{proof}
We rely on an expansion  similar to the one employed by Sinai and rewrite~\eqref{eq:htqFormula} as
\begin{equation} 
\label{eq:Sinai0}
\tq{z ,n} = \ta{z, n} +\sum_{l+m +r =n} \sum_{x, y}
\tq{x ,l}\lb  f^{\theta_x\omega}_{y-x , m} - \fa{y-x ,m}\rb \ta{z-y ,r} .
\end{equation}
This is just a resummation based on the following identity: 
Let $0 <\ell_1 , \dots , \ell_k$ and let $x_1, \dots x_k\in\bbZ^d$. Set $x_0=0$. 
Then,
\begin{equation*}
 \begin{split}
  \prod_1^k f_{x_j -x_{j-1} ,\ell_j}^{\theta_{x_{j-1}}\omega} 
&=
\prod_1^k \fa{x_j - x_{j-1} ,l_j } + \lb \fq{x_1 ,\ell_1} - \fa{x_1 ,\ell_1}\rb
\prod_2^k \fa{x_j - x_{j-1} ,l_j } \\
& + f_{x_1 ,\ell_1}^\omega
\lb f^{\theta_{x_1}\omega}_{x_2 - x_1 ,\ell_2} - \fa{x_2-x_1 ,\ell_2}\rb 
\prod_3^k \fa{x_j - x_{j-1} ,l_j } \\
&  +\dots + 
\prod_1^{k-1} f_{x_j -x_{j-1} ,\ell_j}^{\theta_{x_{j-1}}\omega}
\lb f^{\theta_{x_{k-1}}\omega}_{x_k - x_{k-1} ,\ell_{k}} - \fa{x_{k}-x_{k-1} ,\ell_k}\rb .
 \end{split}
\end{equation*}
\eqref{eq:Sinai0} implies, 
\begin{equation}  
\label{eq:Sinai}
\begin{split}
\tq{n} & = \ta{n} + \sum_{l+m +r =n} \sum_{x }
\tq{x ,l}\lb f^{\theta_x\omega}_{m} - \fa{m}\rb \ta{r} \\
&= \ta{n} +  \frac{1}{\kappa} \sum_{l+m\leq n} \sum_{x }
\tq{x ,l}\lb f^{\theta_x\omega}_{m} - \fa{m}\rb + 
  \sum_{l+m +r =n} \sum_{x }
\tq{x ,l}\lb f^{\theta_x\omega}_{m} - \fa{m}\rb \bigl( \ta{r} -\frac{1}{\kappa}\bigr) \\
& = \frac{1}{\kappa} s^\omega_n  + \lb \ta{n} - 1/\kappa \rb +  \epsilon_n^\omega ,
\end{split}
\end{equation}
where 
\begin{equation}
\label{eq:snterm}
s^\omega_n = 1 + \sum_{l+m\leq n}\sum_{x } 
\tq{x ,l}\lb f^{\theta_x\omega}_{m} - \fa{m}\rb , 
\end{equation} 
and the correction term $\epsilon_n^\omega$ is given by
\begin{equation}  
\label{eq:EpsTerm}
\epsilon_n^\omega = 
 \sum_{l+m +r =n} \sum_{x }
\tq{x ,l}\lb f^{\theta_x\omega}_{m} - \fa{m}\rb \bigl( \ta{r} -\frac{1}{\kappa}\bigr) .
\end{equation}
We claim that, $\bbP$-a.s., 
\begin{equation} 
\label{eq:EpsConv}
\lim_{n\to\infty} s^\omega_n = s^\omega \ \ {\rm and}\ \ 
\sum_n \bbE ( \epsilon^ \omega_n )^2 <\infty .
\end{equation}
The assertion of Theorem~\ref{thm:tqnAsympt} follows.
\qed\end{proof}
The main input for proving~\eqref{eq:EpsConv} is the upper bound of~\eqref{eq:WDBound} and the following maximal inequality of McLeish.

\smallskip
\noindent
\textbf{Maximal Inequality.} Let $X_1 , X_2, \dots $ be a sequence of zero mean and square integrable random variables. Let also $\lbr\calF_k\rbr_{-\infty}^{\infty}$ be a filtration of $\sigma$-algebras. Suppose that we have chosen $\epsilon >0$ and numbers $a_1, a_2 , \dots $ in such a way that
\begin{equation}  
\label{eq:McLeish}
\bbE\lb \bbE\lb X_\ell ~\big| \calF_{\ell -m}\rb^2\rb\leq \frac{a_\ell^2}{(1+m)^{1+\epsilon}} 
\ \ {\rm and}\ \ 
\bbE\lb X_\ell - \bbE\lb X_\ell ~\big| \calF_{\ell + m}\rb\rb^2 \leq 
\frac{a_\ell^2}{(1+m)^{1+\epsilon}} ,
\end{equation}
for all $\ell = 1,2,\ldots $ and $m\geq 0$. Then there exists $K = K (\epsilon ) <\infty$ such that, for all $n_1 \leq  n_2$,
\begin{equation}  
\label{eq:Maximal}
\bbE\max_{n_1\leq m\leq n_2}\Bigl( \sum_{n_1}^m X_\ell\Bigr)^2 \leq K \sum_{n_1}^{n_2} a_\ell^2 .
\end{equation}
In particular, if $\sum_\ell a_\ell^2 <\infty$, then $\sum_\ell X_\ell$ converges $\bbP$-a.s..

\smallskip
\noindent
\textbf{Proof of~\eqref{eq:EpsConv}.}
The difficult part of $s_n^\omega$ in~\eqref{eq:snterm} is 
\begin{equation}  
\label{eq:Xl}
 \sum_{\ell\leq n} \sum_{x } \tq{x ,\ell}\lb\fxq{\, }  -1\rb \df \sum_{\ell = 1}^n X_\ell .
\end{equation} 
To simplify the exposition, let us consider the case of an on-axis external force $h = h\sfe_1$.  By lattice symmetries, the mean displacement $v =\nabla\lambda (h) = v\sfe_1$. At this stage, let us define the hyperplanes $\cHm{m} \df \setof{x}{x\cdot\sfe_1 \leq mv}$ and the $\sigma$-algebras
\[
\calF_m \df \sigma\lb V (x)~:~ x\in\cHm{m }\rb .
\]
Then, 
\[
 \bbE\lb X_\ell ~\big|~\calF_{\ell -m}\rb = 
\sum_{x\in \cHm{\ell - m}} \tq{x}\bbE\lb \fxq{\, } -1
 ~\big|~\calF_{\ell -m}\rb .
\]
Consequently, 
\[
\bbE\Bigl(  \bbE\lb X_\ell ~\big|~\calF_{\ell -m}\rb^2 \bigr) \leq 
\sum_{x ,\xpr \in \cHm{\ell -m}}
\left|
\bbE \tq{x}\tq{\xpr}  
\bbE\lb \fxq{\, } -1
 ~\big|~\calF_{\ell -m}\rb 
\bbE\lb \fxq{\, } -1
 ~\big|~\calF_{\ell -m}\rb 
\right| .
\]
The following notation is convenient: We say that $a_\ell\leqs b_\ell$ if there exists
a constant $c>0$ such that $a_\ell\leqs c b_\ell$ for all $\ell$ . In this language, 
using~\eqref{eq:WDBound}, we bound the latter expression by
\begin{equation}  
\label{eq:VarXlm}
\begin{split}
 &\leqs \frac{1}{\ell^{d}}\sum_{x\in \cHm{ \ell - m}}
e^{-c_2 \abs{x-v\ell}^2/\ell} \sum_{\xpr} e^{-c_2\abs{\xpr -x}}
\leqs  \frac{1}{\ell^{d}} 
\int_{\abs{y} >m} e^{- \abs{y}^2/l}\dd y\\
&= \frac{1}{\ell^{d}} 
\int_m^\infty r^d e^{-r^2 /l}\dd r \leqs 
 \frac{1}{\ell^{(d +1)/2 }}  e^{-m^2 /\ell }. 
\end{split}
\end{equation}
Noting that, for any $\epsilon$ fixed, 
\[
 \frac{e^{-  m^2 /\ell}}{\ell^{1/2 +\epsilon}} \leqs \frac{1}{(1+m)^{1+\epsilon}} ,
\]
we conclude that
\begin{equation}  
\label{eq:Condl-m}
 \bbE\lb  \bbE\lb X_\ell ~\big|~\calF_{\ell -m}\rb \rb^2 \leqs
\frac{1}{\ell^{ (d-1)  /2 -\epsilon}} \frac{1}{(1+m)^{1+\epsilon}} .
\end{equation}
Similarly, the main contribution to $X_\ell  - \bbE\lb X_\ell ~\big|~\calF_{\ell +m}\rb$ comes
from
\[
 \sum_{x\in\calH_{\ell +m}^+} \tq{x ,\ell}\lb\fxq{\, }-1\rb .
\]
By a completely similar computation, 
\begin{equation}  
\label{eq:Condl+m}
 \bbE\lb   X_\ell  - \bbE\lb X_\ell ~\big|~\calF_{\ell +m}\rb \rb^2 \leqs
\frac{1}{\ell^{ (d-1)/2 -\epsilon}}\cdot \frac{1}{(1+m)^{1+\epsilon}} .
\end{equation}
Therefore, \eqref{eq:McLeish} applies with $a_\ell^2 \eqvs \ell^{-(d-1)/2 -\epsilon}$. 
Since  $d\geq 4$, we rely on \eqref{eq:Maximal} and deduce \eqref{eq:EpsConv}.

\subsection{Quenched CLT.}  
One possible strategy for proving a $\bbP$-a.s CLT would be to try to adjust a powerful approach by
Bolthausen-Sznitman~\cite{BolthausenSznitmanCLT} which was developed in the context of ballistic 
RWRE. It appears, however, that a direct work on generating functions goes through. 
Let us introduce 
\[
\qS{n}{\omega}  (\alpha )\df 
\sum_{z} \ta{z ,n} e^{i \alpha \cdot (z - nv )/\sqrt{n}} .
\]
The asymptotics of $\aS{n} (\alpha) \df \bbE \qS{n}{\omega} (\alpha )$ are given in~\eqref{eq:AnCLT}.  Using~\eqref{eq:Sinai}, we can write
\begin{equation}  
\label{eq:S1}
\qS{n}{\omega}  (\alpha  ) = \aS{n} (\alpha  ) + 
\sum_{l+m +r =n} \sum_{x, y , z }
\tq{x ,l}\lb f^{\theta _x\omega}_{y-x, m} - \fa{y-x, m}\rb \ta{z- y, r}
e^{i \alpha  \cdot (z - nv )/\sqrt{n}} .
\end{equation}
Define $\alpha_n^r \df \alpha \sqrt{r/ n}$ and 
\begin{equation}  
\label{eq:GFunction}
 G^\omega_m (\alpha ) \df \sum_{y}e^{(y -mv )\cdot i\alpha }\lb \fq{y ,m }- \fa{m}\rb .
\end{equation}
We can rewrite~\eqref{eq:S1} as  
\begin{equation}  
\label{eq:S2}
\begin{split}
\qS{n}{\omega}  (\alpha  ) &= \aS{n} (\alpha  ) + 
 \sum_{l+m +r =n} 
\aS{r}\lb \alpha_n^r \rb
 \sum_{x}
\tq{x ,\ell}
e^{(x-\ell v)\cdot i\frac{\alpha}{\sqrt{n}}} G_m^{\theta_x\omega} \bigl( 
\frac{\alpha}{\sqrt{n}}\bigr) \\
&= \aS{n} (\alpha  )\Bigl( 1 + \sum_{\ell +m \leq n}\sum_x \tq{x ,\ell}\lb \fxq{m } - \fa{m}
\rb\Bigr)\\
&+  \sum_{l+m +r =n} \bigl( \aS{r} (\alpha_n^r ) - \aS{n} (\alpha )\bigr)
\sum_x \tq{x ,\ell}\lb \fxq{m } -\fa{m}\rb \\
&+ 
\sum_{l+m +r =n}  \aS{r} (\alpha_n^r ) 
\sum_x \tq{x ,\ell} \Bigl( G^{\theta_x \omega}_m  \bigl( \frac{\alpha}{\sqrt{n}}\bigr) - 
G^{\theta_x \omega}_m (0)\Bigr)
\\
&+ 
\sum_{l+m +r =n}  \aS{r} (\alpha_n^r )
\sum_x \tq{x ,\ell}
\bigl( e^{(x -\ell v)\cdot i\alpha/\sqrt{n}} -1\bigr) 
G^{\theta_x \omega}_m \bigl( \frac{\alpha}{\sqrt{n}}\bigr)\\
& \df \aS{n} (\alpha  ) s_n^\omega + \sum_{i=1}^3 \hat \calE_n^i(\omega ) ,
\end{split}
\end{equation}
where $s_n^\omega$ is as in~\eqref{eq:snterm}. 
\begin{theorem}
\label{thm:QuCLT}
For every $\alpha\in\bbR^d$, the correction terms $\hat \calE_n^i(\omega )$ in~\eqref{eq:S2} satisfy
\begin{equation} 
\label{eq:CortermsCLT}
\text{For $i=1,2,3$, }\ \ \lim_{n\to\infty} \hat \calE_n^i(\omega ) = 0\,, \ \text{$\bbP$-a.s.\ and 
in $L_2 (\Omega )$} .
\end{equation}
\end{theorem}
The proof of Theorem~\ref{thm:QuCLT} is technical and will appear elsewhere~\cite{IoffeVelenik-inprogress}. In view of~\eqref{eq:AnCLT} and~\eqref{eq:tqnAsympt}, the convergence in~\eqref{eq:CortermsCLT} implies that
\begin{equation} 
\label{eq:QuCLT} 
\lim_{n\to\infty} \frac{\qS{n}{\omega}  (\alpha )}{\tq{n}} = \exp\bigl\{ - \frac12\Xi^{-1}\, \alpha \cdot\alpha\bigr\} ,
\end{equation}
$\bbP$-a.s.\ for every $\alpha\in\bbR^d$ fixed. 

\section{Strong Disorder}
\label{sec:strong} 
In this section, we do not impose any moment  assumptions on the environment $\lbr V (x) \rbr$.  Even the case of traps (i.e., when $\bbP \lb V =\infty \rb >0$) is not excluded. We still need that $\bbP\lb V\neq 0,\infty \rb >0$. Without loss of generality, we shall assume that $\bbP \lb V\in (0,1 ]\rb  >0$. Under this sole assumption, the environment is always strong in two dimensions in the following sense.
\begin{theorem}
\label{thm:Strong}
Let $d=2$ and $\beta , \lambda >0$. There exists $c = c(\beta ,\lambda ) >0$ such that the following holds: For any $x\in\bbS^1$ define $x_n =\lfloor nx\rfloor$. Then, 
\begin{equation}  
\label{eq:Strong}
\limsup_{n\to\infty} \frac{1}{n}\log\frac{Q_\lambda (x_n )}{A_\lambda (x_n )} < -c.
\end{equation}
In particular, $\fra_\lambda <\frq_\lambda$ whenever $\frq_\lambda$ is well defined.
\end{theorem}
\begin{remark}
As in~\cite{Lacoin} and, subsequently, \cite{Zygouras-StrongDisorder} proving strong disorder in dimension 
$d=3$ is a substantially more delicate task.
\end{remark}
Let us explain Theorem~\ref{thm:Strong}:  By the exponential Markov's inequality 
(and Borel-Cantelli) it is sufficient to prove that there exist $c^\prime >0$ and  $\alpha >0$ such that 
\begin{equation} 
\label{eq:FracMom}
\bbE 
\lb \frac{Q_\lambda (x_n )}{A_\lambda (x_n )}\rb^\alpha  \leq e^{-c^\prime  n} .
\end{equation}
\textbf{Normalization.}
In order to facilitate the notation we shall proceed with an on-axis case $x = (0,x)$.
Let $h = (0,h) \in\partial\Kal{\lambda}$ be unambiguously defined by the relation $\fra_\lambda (x ) =h\cdot x$.  We shall explore
\[
\frac{1}{n}\log\frac{Q_{\lambda ,h} (x_n )}{A_{\lambda, h} (x_n )}
\df 
\frac{1}{n}\log
\frac{e^{h\cdot x_n}Q_\lambda (x_n )}{e^{h\cdot x_n }A_\lambda (x_n )} .
\]
Since the annealed Lyapunov exponent $\fra_{\lambda}$ is well-defined, we can rely on the logarithmic equivalence  $A_{\lambda, h} (x_n ) \asymp 1$.
Note that, for any family of paths $\Gamma_n$, 
\[
 A_{\lambda, h} \lb X(\gamma ) = x_n ;\Gamma_n\rb \df 
 A_{\lambda, h} \lb  x_n ;\Gamma_n\rb = \bbE Q_{\lambda, h} \lb  x_n ;\Gamma_n\rb .
\]
Consequently, by the exponential Markov inequality and Borel-Cantelli Lemma,  we can ignore the families $\Gamma_n$ for which $A_{\lambda, h} \lb X(\gamma ) = x_n ;\Gamma_n\rb\leq e^{-c\abs{x_n}}$. 

\smallskip
\noindent
\textbf{Reduction to Basic Partition Functions.}
In particular, we can restrict attention to paths $\gamma$ which have at least two cone points. With a slight abuse of notation,
\[
Q_{\lambda, h} \lb  x_n\rb = \sum_{y ,z} f^\omega_{>} (y) 
t^{\theta_y\omega}_{z-y}f^{\theta_z \omega}_{<} (x_n -z) ,
\]
where $f^\omega_{>}$ and $f^{\theta_z \omega}_{<}$ are the weights of the 
initial and the final irreducible pieces $\omega_{>}$ and $\omega_{<}$ in 
the decomposition \eqref{eq:DecompositionIntoIrreduciblePieces}. 
The left and right irreducible partition functions satisfy $\bbE  f^\omega_{>} (u), 
\bbE  f^\omega_{<} (u)
 \leq e^{-c \abs{u}}$. Consequently, for fractional powers $\alpha\in (0,1)$, 
\[
\bbE\lb  Q_{\lambda, h} \lb  x_n\rb \rb^\alpha \leq 
\sum_{y ,z}
e^{-c\abs{y}}   
\bbE \lb \tq{z-y}\rb^\alpha 
e^{-c\abs{x_n -z}} .
\]
As a result we need to check that
\begin{equation}  
\label{eq:tterm}
\limsup_{u\to\infty}\frac{1}{\abs{u}}\bbE \lb \tq{u}\rb^\alpha < 0.
\end{equation}
In its turn, \eqref{eq:tterm} is routinely implied by the following  statement (\eqref{eq:rterm} below): Let $\rrq{N}$ be the partition function of $N$ irreducible steps:
\begin{equation}  
\label{eq:pfSteps}
\rrq{N} = \sum_x \rrq{x,N}  \df \sum_x \sum_{u_1,\cdots, u_{N-1}} f^\omega_{u_1} f^{\theta_{u_1}\omega}_{u_2 -u_1}
\cdots 
f^{\theta_{u_{N-1}}\omega}_{x -u_{N-1}} .
\end{equation}
Then, 
\begin{equation} 
\label{eq:rterm}
\limsup_{N\to\infty} 
\frac{1}{N}
\log\bbE\lb \rrq{N}\rb^\alpha < 0 .
\end{equation}
Note by the way that by the very definition of the irreducible decomposition any trajectory $\gamma$ which contributes to a $\underline{u}=(u_0 , u_1, \dots, u_N )$-term in \eqref{eq:pfSteps} is confined to the set $D (\underline{u} ) = \cup_\ell D(u_{\ell - 1} ,u_\ell)$ where the diamond shape $D(u_{\ell - 1} ,u_\ell)\df \lb u_{\ell - 1} + \fcone\rb\cap\lb u_\ell +\bcone\rb $.

\smallskip
\noindent
\textbf{Fractional Moments.}
Following Lacoin~\cite{Lacoin}, \eqref{eq:rterm} follows once we show that there exist $N$ and $\alpha\in (0,1)$ such that
\begin{equation} 
\label{eq:target}
\bbE\sum_x \lb \rrq{x ,N}\rb^\alpha <1 .
\end{equation}
Pick $K$ sufficiently large and $\epsilon$ small, and consider
\[
A_N = \{0, \ldots ,KN\} \times \{-N^{1/2 +\epsilon}, \ldots,  N^{1/2 +\epsilon}\} \subset\bbZ^2 .
\]
Since $\bbE \bigl(\rrq{x ,N}\bigr)^\alpha \leq \bigl( \bbE \rrq{N,x} \bigr)^\alpha \df  r_{N,x}^\alpha$, annealed estimates enable us to restrict attention to $x\in A_N$. Furthermore, since, as was explained in Subsection~\ref{ssec:OZ}, $r_N$ is a partition function which corresponds to an effective random 
walk with an on-axis drift and exponential tails we may restrict attention only to the effective trajectories $\underline{u}$ which satisfy $D (\underline{u} )\subseteq A_N$. By the confinement property of the irreducible decomposition we may therefore restrict attention to microscopic polymer configurations $\gamma$ which stay inside $A_N$.

At this stage, we shall modify the distribution of the environment inside  $A_N$ in the following way: The modified law of the environment $\tilde\bbP$ is still product and, for every $x\in A_N$,
\[
 \frac{\dd \tilde\bbP}{\dd \bbP}
\lb V \rb 
\df e^{-\delta_N \lb V\wedge 1\rb + g (\delta_N )},  \quad 
\text{where $e^{- g (\delta )} = \bbE e^{-\delta \lb V\wedge 1\rb }$.}
\]
From Hölder's inequality,
\[
 \bbE \lb \rrq{x ,N}\rb^\alpha \leq \Bigl( \tilde \bbE\bigl(  
\frac{\dd \tilde\bbP}{\dd \bbP}\bigr)^{1/(1-\alpha ) }\Bigr)^{1-\alpha}
\bigl( \tilde\bbE\, \rrq{x ,N} \bigr)^\alpha .
\]
Now, the first term is 
\[
\Bigl( \tilde\bbE\, \bigl(
\frac{\dd \tilde\bbP}{\dd \bbP}
\bigr)^{1/(1-\alpha ) }\Bigr)^{1-\alpha} 
= 
\Bigl( \tilde\bbE\, \exp\Bigl\{ \Bigl( \delta_N (V\wedge 1) - g (\delta_N )\Bigr)/(1-\alpha )\Bigr\}
\Bigr)^{(1-\alpha )\abs{A_N}} .
\]
However, the first order terms in $\delta_N$ cancel,
\begin{equation} 
\label{eq:quadratic}
\begin{split}
&\tilde\bbE\,  \exp\Bigl\{ \frac{\delta_N (V\wedge 1) - g (\delta_N )}{1-\alpha}\Bigr\}
 = 
\bbE\,  \exp\Bigl\{ \frac{\alpha}{1-\alpha} \bigl(\delta_N (V\wedge 1) - g (\delta_N ) \bigr)\Bigr\} \\
& \quad = \exp\Bigl\{ - g \bigl( -\frac{\alpha}{1-\alpha} \delta_N \bigr) - 
\frac{\alpha}{1-\alpha}g\lb \delta_N \rb \Bigr\} 
\leq \exp\Bigl\{ \frac{\alpha}{(1-\alpha^2 )^2}\delta_N^2 \Bigr\} .
\end{split}
\end{equation}
On the other hand (recall that $\calF$ is the set of irreducible paths), 
\[
 \tilde\bbE \rrq{x, N} \leq \tilde \bbE\rrq{N} = 
\Bigl( \tilde \bbE\sum_{\gamma\in\calF} q_{\lambda ,h} (\gamma )\Bigr)^N \df 
\tilde f (\delta_N )^N .
\]
It is straightforward to check that $\tilde f^\prime (0) <0$. As a result, $\tilde \bbE\rrq{N} \leq e^{-c\delta_N}$. 

\noindent
We are now ready to specify the choice of $\delta_N$. We want to have simultaneously
\[
\delta_N^2 \abs{A_N} \ll \delta_N N \quad\text{and}\quad \delta_N N\gg N^\epsilon .
\]
The choice $\delta_N = N^{-1/2 - 2\epsilon}$ with $\epsilon \in (0,1/3 )$ qualifies, and~\eqref{eq:target} follows.\qed

\bigskip
\textbf{Acknowledgments.} The research of DI was supported by the Israeli Science Foundation
(grant No. 817/09). YV is partially supported by the Swiss National Science Foundation.

\end{document}